\documentclass[11pt]{article}
\usepackage{palatino}
\usepackage{aliascnt}
\usepackage{amsmath}
\usepackage{amsrefs}
\usepackage{amssymb}
\usepackage{amsthm}
\usepackage{algorithm}
\usepackage{algpseudocode}
\usepackage{tabularx}

\usepackage{bm}
\usepackage{bbm}

\usepackage{calrsfs}
\usepackage{color}

\usepackage[sans]{dsfont}

\usepackage{enumerate}
\usepackage{epsfig}
\usepackage{extarrows}
\usepackage[mathscr]{euscript}
\usepackage[symbol]{footmisc}
\usepackage{framed}
\usepackage{fullpage}
\usepackage[margin=1in]{geometry}
\usepackage{graphicx}
\usepackage{import}
\usepackage{multirow}

\usepackage{paralist}
\usepackage{psfrag}

\usepackage{subfig}
\usepackage[]{times}
\usepackage{tikz}
\usepackage{transparent}
\usepackage[normalem]{ulem}
\usepackage{url}
\usepackage{verbatim}
\usepackage{wrapfig}
\usepackage{xcolor}
\usepackage{xifthen}
\definecolor{Green}{rgb}{0.0,0.45,0.0}
\usepackage[colorlinks,linkcolor=blue,citecolor=red,bookmarksopen=false]{hyperref}

\def\NewTheorem#1#2{%
	\newaliascnt{#1}{thmm}
	\newtheorem{#1}[#1]{#2}
	\aliascntresetthe{#1}
	\expandafter\def\csname #1autorefname\endcsname{#2}
}

\numberwithin{equation}{section}

\NewTheorem{lem}{Lemma}
\NewTheorem{thm}{Theorem}
\NewTheorem{cor}{Corollary}
\NewTheorem{prop}{Proposition}
\theoremstyle{definition}
\NewTheorem{defin}{Definition}
\theoremstyle{remark}
\NewTheorem{remark}{Remark}
\theoremstyle{definition}
\NewTheorem{eg}{Example}

\newcommand{\sign}{\text{sign}}

\renewcommand{\P}{ \mathbb P }
\newcommand{\EXP}{\mathbb{E}}
\newcommand{\IND}[1] {{ \mathds{1}_{ #1 }} }
\newcommand{\1}{\mathbbm{1}}

\newcommand{\R}{\mathbb{R}}
\newcommand{\C}{\mathbb{C}}

\newcommand{\X}{\mathbf{X}}
\newcommand{\barX}{\bar{\mathbf{X}}}

\newcommand{\T}{\mathbb{T}}

\newcommand{\bX}{{\bf{X}}}

\newcommand{\Rd}{\mathbb{R}^d}
\newcommand{\mirror}[1]{%
  \mathord{%
    \mathchoice
      {\reflectbox{$\displaystyle #1$}}%
      {\reflectbox{$\textstyle #1$}}%
      {\reflectbox{$\scriptstyle #1$}}%
      {\reflectbox{$\scriptscriptstyle #1$}}%
  }%
}
\renewcommand{\root}{{\mirror{\emptyset}}}

\newcommand{\Vo}{\mathring{V}}

\newcommand{\bY}{\bold{Y}}
\newcommand{\bZ}{\bold{Z}}

\title{Navier-Stokes Equations with Fractional Dissipation and Associated Doubly Stochastic Yule Cascades}

\author{} 

\author{Radu Dascaliuc\thanks{Department of Mathematics,  Oregon State University, 
Corvallis, OR, 97331. {dascalir@math.oregonstate.edu}}
\and Tuan N.\ Pham\thanks{Faculty of Math and Computing,  Brigham Young University-Hawaii,  
Laie, HI 96762. 
{tpham@byuh.edu}}
\and Enrique Thomann\thanks{Department of Mathematics,  Oregon State University, 
Corvallis, OR, 97331.
{thomann@math.oregonstate.edu}}
\and
Edward C.\ Waymire\thanks{Department of Mathematics,  Oregon State University, Corvallis, OR, 97331. 
{waymire@math.oregonstate.edu}}
}

\begin{document}
\maketitle

\begin{abstract} 
We introduce a self-similar doubly stochastic Yule (DSY) cascade associated with the deterministic Navier-Stokes equations (NSE) in $\R^d$ with fractional dissipation $(-\Delta)^\gamma$. Interestingly, such a structure is well-defined only in the scaling-supercritical regime $\gamma\in(\frac{1}{2},\frac{d+2}{4})$. We then characterize parametric regions of $(d,\gamma)$ that correspond to the stochastically explosive, non-explosive, hyperexplosive, non-hyperexplosive behaviors of the DSY cascade. Stochastic solution processes are constructed recursively to randomly
transform (deterministic) initial data, such that their expectations yield solutions to the fractional NSE whenever these expectations exist. Explosion and geometric properties of the DSY cascade are then exploited to establish non-uniqueness and finite-time blowup results for a scalar partial differential equation associated with the fractional NSE using a majorization principle for stochastic solution processes. In the special case $d=2$, we derive a closed form for the solution process and prove the finite-time loss of integrability of the solution process for sufficiently large initial data. This lack of integrability does not necessarily imply finite-time blowup of solutions to the fractional NSE. Indeed, for vortex-flow initial data, we show that the solution can be continued beyond the time of integrability breakdown by averaging the stochastic solution processes in a way that creates symmetry cancellations.


\end{abstract}
\tableofcontents
\section{Introduction}

The natural scaling property of the three-dimensional incompressible 
Navier-Stokes equations 
\begin{equation}\label{nse}\tag*{(NSE)}u_t-\Delta u+(u\cdot\nabla) u+\nabla p=0,
\ \text{div}\,u=0, x\in\R^3,t\ge 0,
\end{equation}
plays an important role in the unresolved problem of 
global well-posedness, e.g., see \cite{colomboicm}. In essence this is the invariance 
of Newton's (fluid) law of motion
 under re-scaling by $\lambda>0$ of its instrinsic
 velocity and pressure
terms, respectively, $u=u(x,t)$ and $p=p(x,t)$, by:
\begin{equation}\label{R3-scaling}
u_\lambda(x,t)=\lambda u(\lambda x,\lambda^{2}t),\ \ 
p_\lambda(x,t)=\lambda^{2}p(\lambda x,\lambda^{2}t),\ \ u_{0\lambda}(x)=\lambda u_0(\lambda x).
\end{equation}
 Such scaling invariance naturally leads to
the following notion of  {\it scaling criticality} for the Navier-Stokes equations: A
 solution-dependent functional $Q=Q(u,p)$ is said to be
 {\em {scaling critical}} if it has the invariance property
$Q_\lambda=Q(u_\lambda,p_\lambda)=Q$ for all $\lambda>0$. 
On the other hand, if  
$Q_\lambda\to0$ as $\lambda\to\infty$, 
then $Q$ is said to be  {\em supercritical}, while if $|Q_\lambda|\to\infty$ 
as $\lambda\to\infty$, then $Q$ is called {\em subcritical}.
One of the well known features of the well-posedness problem for the Navier-Stokes
equations in three dimensions is the supercritical nature of the physically 
relevant functionals. As a classical example, the kinetic energy of the 
solution given by its $L^2$-norm is a-priori bounded: 
$\|u(\cdot,t)\|^2_{L^2}\le \|u(\cdot,0)\|^2_{L^2}$ for all $t\ge 0$.
Since $\|u_\lambda\|^2_{L^\infty_tL^2_x}=\lambda^{-1}\|u\|^2_{L^\infty_tL^2_x}\to0$ as $\lambda\to\infty$, 
the energy is a supercritical quantity. However, in order to establish the 
well-posedness and regularity of the solutions, one needs 
to establish boundedness of the enstrophy, i.e. of $\|\nabla u(\cdot,t)\|^2_{L^2}$, 
which is a subcritical quantity: 
$\|\nabla u_\lambda\|^2_{L^\infty_tL^2_x}=\lambda\|\nabla u\|^2_{L^\infty_tL^2_x}$.
This scaling gap is recognized as a major obstacle to resolving the problem of global 
in time existence and uniqueness of the strong solutions, \cites{bradshaw2019algebraic,lemarie2002recent,tao09,cheskidov2022sharp}.

Another manifestation of this scaling gap arises when considering the 
Navier–Stokes equations on $\R^d$, $d\ge2$, 
with fractional dissipation $\gamma>0$:
\begin{equation}\label{nsedg}\tag*{$\textup{(NSE)}_{d,\gamma}$}u_t+(-\Delta)^\gamma u+u\nabla u+\nabla p=0,\ \text{div}\,u=0,\ 
u(\cdot,0)=u_0,
\end{equation}
where $u=u(x,t)$, $x\in\Rd$, $t\ge 0$. A natural scaling that preserves the system is
\begin{equation}\label{scaling}
u_\lambda(x,t)=\lambda^{2\gamma-1}u(\lambda x,\lambda^{2\gamma}t),\ \ p_\lambda(x,t)=\lambda^{4\gamma-2}p(\lambda x,\lambda^{2\gamma}t),\ \ u_{0\lambda}(x)=\lambda^{2\gamma-1}u_0(\lambda x).\ \ 
\end{equation}
On and above the scaling critical line 
\[
\gamma=\frac{d+2}{4}
\]
in the $(d,\gamma)$ plane, there are a number of global well-posedness and regularity results, including the classical $2$-dimensional Navier–Stokes equation 
\cites{lady,constantin_foias,katz,tao09,nan,baaske,ferreira,colombo}. In contrast, below the critical line, that is, in the supercritical regime $\gamma<(d+2)/4$, the questions of global well-posedness 
and regularity remain largely open, with some evidence suggesting finite-time blowup and non-uniqueness of solutions 
\cite{cordoba,DA_EB_MC2022,de2019infinitely}. This supercritical regime includes the  most prominent three-dimensional case with $\gamma=1$. 
Incidentally, the critical case $\gamma=(d+2)/4$ corresponds precisely to the situation in which the energy is scaling-invariant; see \autoref{Section2} and in particular \eqref{energy}.

 While the approach taken in this paper to analyze uniqueness/nonuniqueness of solutions to \ref{nsedg}
 fully recognizes scaling criticality in the
 sense described above, it alternatively rests on new notions of {\it statistical criticalities} 
 and {\it cancellation properties} defined by a stochastic model for which \ref{nsedg} describes its mean flow. 
Notably, starting with \ref{nsedg},
this stochastic structure arises 
directly from the natural  scaling of the Fourier-transformed 
fractional Navier–Stokes equations in the mild form. That
is, the Fourier transform $\hat{u}$ has the same spatial scaling property as
\[
h(\xi) = \frac{c_{d,\gamma}}{|\xi|^{d+1-2\gamma}}.
\] 
Scaling the value of $\hat{u}(\xi,t)$ by $h$ transforms the mild formulation of \ref{nsedg} into an integral equation \eqref{mildnse} of $\chi=c\hat{u}/h$ with an integral kernel $H$ given by 
\eqref{Hkernel}. This kernel can also be obtained by rescaling a self-similar solution to one that is defined on the unit sphere of $\R^d$ (see \cite[Sec.\ III]{chaos} for the case $d=3,\,\gamma=1$). Remarkably, $H$ is a probability kernel yielding a stochastic representation of 
solutions precisely in the scaling supercritical regime $1/2 < \gamma < (d+2)/4$.
In this regime, the rescaled form \eqref{mildnse} of \ref{nsedg} can be interpreted 
as a mean flow equation whose stochastic structure 
is defined by the following two random fields: 

(i) A {\it doubly stochastic Yule (DSY) cascade} $\bY$ indexed by the vertices of a full binary tree, labelled by the (random)
wave numbers. As discussed in \autoref{Section3},  
along each fixed ray emanating from the root the labels  form a Markov process with transition probabilities (marginally) defined
by $H$, and 

(ii) a {\it stochastic solution process} $\bX=\bX(\xi,t)$, defined by a binary (multiplicative) recursion initiating from the
Fourier transformed initial data along
the evolution of $\bY$ on a common probability space $(\Omega, \mathcal{F}, \P)$ as discussed in \autoref{Section2}. 

In particular, a solution process $\bX$ extracts this 
salient random field structure
of the evolution to 
randomly transform initial data in such a way that its
expected value, if it exists, defines a solution $\chi(\xi,t)=\EXP_{\xi}\X(\xi,t)$ to the mean flow equation \eqref{nFMS} for the given initial data. As will be
seen, such a representation provides a probabilistic framework for analyzing various well-posedness 
problems for the mean flow equation.

Apart from its extension to fractional Navier-Stokes equations, 
this statistical framework is most notably a generalization of the framework of Le Jan and Sznitman \cite{lejan}
in which a coin-tossing (thinning) mechanism was used to eliminate the possibility of having infinitely many branch splittings in finite time; a critical phenomenon to be referred to as \emph{stochastic explosion}. 

In the event of explosion and in the absence of
a thinning mechanism, it was
 recognized in \cite{alphariccati, dascaliuc2019jan} that
  there is 
 a well-defined \emph{minimal solution process} for all $t>0$ which, in turn, defines a solution $\chi(\xi,t)$ to the mean flow equation up to the first time of integrability breakdown:
\[t_c(\chi_0) = \sup\{t\ge 0: \mathbb{E}_\xi|{\mathbf{X}}(\xi,t)|<\infty\text{~~a.e.~~}\xi\in\R^d\}.\]
This is referred to as the \emph{critical time of integrability} for initial data $\chi_0=c\hat{u}_0/h$. 
Results of \cite{chaos, part1_2021, part2_2021, alphariccati}, illustrate the use of 
various statistically defined
critical phenomena 
 as a means to establish non-uniqueness and finite-time blowup of solutions \cites{alphariccati,transformation2025,dascaliuc2019jan}. 

The probabilistic representation of solutions naturally defines a global solution to \ref{nsedg} in the scaling critical setting 
with small initial data (\autoref{pmsol}). This is an extension of the global well-posedness result obtained by \cite{lejan} for $d=3,\,\gamma=1$ (see also \cite[Sec.\ 24.1]{lemarie2002recent} for an analytic approach). More recently, 
\cite{ferreira} proved a number of local and global well-posedness results for \ref{nsedg}, $\gamma\in(\frac{1}{2},\frac{d+2}{4})$, in similar settings. 
For the fractional Navier-Stokes equations on the three-dimensional torus, interested readers can find various regularity results depending on the range of the dissipation power $\gamma$ in \cite{BoutrosGibbon}. 

In addition to the fractional Navier-Stokes equations in $\Rd$,  its scalar counterpart in $\Rd$, the fractional Montgomery–Smith equation given by \ref{MS}, is also central to the
present paper.  These equations arise naturally by simply
 restricting $\hat{u}$ to scalar functions on $\R^d$ with
 ordinary multiplication, but otherwise
 share the same DSY cascade $\bY$.
It will be shown that $\bY$ can exhibit 
a variety of non-explosive, explosive, non-hyperexplosive, or hyperexplosive critical behaviors to be defined,  depending on the relation between the dimension $d$ and $\gamma$ (\autoref{dgdiagram}). Broadly throughout statistical physics,  parametric regimes corresponding to qualitative changes in statistical properties of the evolution of a stochastic structure may be viewed as the \emph{phase transition} phenomena underlying the behavior of the expected mean flow \cite{dean2005phase, edengrowth}. This is indeed the case for the scalar \ref{MS}, where the explosion regime leads to non-uniqueness of solutions (\autoref{nonuniquenessMS}). Therefore, as the DSY cascade transitions from being non-explosive to explosive, the uniqueness of solutions to the fractional Montgomery-Smith equation transitions to non-uniqueness.
Furthermore, we show that the solution of \ref{MS} blows up in finite time for sufficiently large compactly supported initial data, thus extending a prior result \cite[Prop.\ 4.14]{dascaliuc2019jan} to the entire supercritical range $d\ge 1$, $\gamma\in(\frac{1}{2},\frac{d+2}{4})$ (\autoref{BlowupMSdgamma}). 
We also show that, in the scaling critical setting, there is no minimal threshold for the magnitude of the initial data that guarantees blow up of solutions. In other words, there are no minimal blowup initial data for \ref{MS} (\autoref{mspmsol}).


For the fractional Navier-Stokes equations, stochastic 
explosion or hyperexplosion alone does not appear to qualitatively strengthen existing existence and uniqueness results for mild solutions.  In fact, \autoref{blowup} shows that for $d=2$ and any $\gamma\in(1/2,1)$, a range that includes both hyperexplosion and explosion regimes (see \autoref{dgdiagram}), the solution process has a finite critical time of integrability $t_c=t_c(\chi_0)$. The minimal stochastic solution process represents a solution to \ref{nsedg}  via its expected value only when $t<t_c$. However, this need not imply finite time blow-up of solutions.


Perhaps counterintuitively, the explosive and hyperexplosive regimes for 
$\bY$ occur near the scaling critical line $\gamma=(d+2)/4$ (see \autoref{dgdiagram}), 
where existence and uniqueness results are known (see e.g.\ \cite{tao09,nan,colombo}). By contrast, the non-explosive 
regime lies deeper in the supercritical region, where only partial regularity results are available (see e.g.\ \cite{katz, cordoba}). This highlights the limitation of interpreting 
stochastic explosion of the DSY cascade as a direct metaphor 
for ill-posedness.

In the two-dimensional case ($d=2$), a closed-form expression for the minimal solution 
process $\bX$ (\autoref{Xclosedform}) is obtained 
using the simpler geometric structure of the nonlinear term in the two-dimensional case. Using this form, we are able to construct a compactly supported initial data $\chi_0$ for which $t_c(\chi_0)<\infty$. An important implication is that for any initial data $\chi_0$ that is sufficiently large on an annulus $\{\xi:\,4r<|\xi|<7r\}$ for some $r>0$, the integrability of $\bX$ is lost after a finite time (\autoref{blowup}). The loss of integrability of the solution process does not necessarily imply the finite-time blowup of solutions to the fractional Navier-Stokes equations. Indeed, for vortex-flow initial data $u_0$ (i.e.\ velocity field in which all streamlines are concentric circles), in \autoref{Section5}  it is
shown that the solution can be continued beyond $t_c$ by averaging the stochastic solution processes in a way that creates symmetry cancellations (\autoref{radial1}). 

While it is classical that two-dimensional vortex-flow initial data lead to global solutions via cancellations in the nonlinearity (see e.g.\ \cite[Example 2.2, p.\ 48]{majda}, \cite[Sec.\ 4.5]{batchelor}, \cite[Chap.\ 13]{saffman}), the result presented here is the first to reveal 
this cancellation directly at the level of the stochastic cascade. This perspective illustrates the possibility of using a “mean evolution” of $\bX$ to bypass the limitations of the \emph{majorizing principle} to obtain the integrability of the solution process (see \autoref{reduction}), a method that has played a central role in earlier existence proofs \cite{lejan, dascaliuc2019jan}. 

Nevertheless, a key challenge remains: 
There is a discrepancy between the loss of 
integrability of $\bX$ (expressed by the finiteness of $t_c$) and the finite-time blowup of 
the mean flow equation. Resolving this gap for general initial data requires a more general method to continue the solution past $t_c$, perhaps by introducing a new 
notion of averaging for $\bX$ that exploits cancellations due to symmetry in both the solution process $\bX$ and the DSY cascade $\bY$. 
 While \ref{nsedg} is a fully  deterministic
linear parabolic equation,  it should perhaps not go
unnoticed that there is some similarity between
the probabilistic methods being pursued here
with the powerful
 cancellation and tree methodology of \cite{deng} 
 for example, and 
 references therein, 
for dispersive equations having random (Gaussian) initial data.

The paper is organized as follows.  In \autoref{Section2}, we formulate a probabilistic representation of the solution to the incompressible fractional Navier-Stokes equations in the Fourier space (\autoref{819191}). This construction involves a solution process $\bX$ recursively defined  on a stochastically labeled binary tree $\bY$
by \eqref{eq:822191}. In \autoref{Section3}, we analyze this binary tree, which is a \emph{self-similar} doubly stochastic Yule cascade according to \cite{part2_2021}. We define the 
statistically critical events of non-explosion, explosion, non-hyperexplosion, and hyperexplosion, and characterize the corresponding behaviors of $\bY$ in terms of the dimension $d$ and the fractional 
exponent parameters $\gamma$ (\autoref{exp-nonexp} and \autoref{dgdiagram}).  In \autoref{reduction}, the {\it Montgomery-Smith} 
model is introduced as a tool to estimate the critical time $t_c$, which informs the integrability of the solution process, by a comparison method (called \emph{majorizing principle}). A more general majorizing principle can be found in \cite{dascaliuc2019jan}. \autoref{Section5} is dedicated to the special case $d=2$, where the stochastic geometry and symmetries underlying the 
solution process $\bX$ are significantly simplified while all critical phenomena except for non-explosion still exhibit depending on the range of $\gamma$. We show that the integrability of $\bX$ breaks down for any sufficiently large initial data (\autoref{blowup}). A method to extend the solution beyond the critical time $t_c$ for vortex-flow initial data is given in \autoref{radial1}. Interestingly, 
it also shows that the locally defined solution extends to a global solution. 

\section{The Fractional Navier-Stokes Equation as a Mean Flow Equation}\label{Section2}
As noted in the introduction standard notions of ``scaling criticality" exist in the PDE literature, and ``statistical criticality"
 in probability/statistical physics literature.  
While the focus of the present paper is on statistical
criticalities,  let us 
 compute the scaling critical regimes for \ref{nsedg}
 for points of reference.
Let
\begin{equation}\label{energy}
e(u,t)=\frac{1}{2}\int_{\Rd}|u(x,t)|^2dx+\int_0^t\int_{\Rd}|(-\Delta)^{\gamma/2}u(x,s)|^2dxds.
\end{equation}
Then one has
 the energy estimate $e(u,t)\le \frac{1}{2}\int_{\Rd}|u_0|^2dx$ for all $t\ge 0$.
 Under the natural scaling \eqref{scaling}, the energy is scaled as $e(u_\lambda,t)=\lambda^{4\gamma-(d+2)}e(u,\lambda^{2\gamma}t)$. 
 With this, the system \ref{nsedg} is 
  scaling critical if $\gamma=\frac{d+2}{4}$ (i.e.\ if the energy is scale-invariant), scaling supercritical if $\gamma<\frac{d+2}{4}$, and scaling subcritical if $\gamma>\frac{d+2}{4}$. The system is called scaling hypodissipative for $\gamma<1$ and scaling hyperdissipative for $\gamma>1$. Global regularity is known for the scaling critical and scaling subcritical regime $\gamma\ge\frac{d+2}{4}$ (see e.g.\ \cite{katz,tao09,nan}). Local and global regularity holds for mild solutions for  $\gamma\in(\frac{1}{2},\frac{d+2}{4})$ and $u_0\in PM^{a,b}:= PM^a+PM^b$ for certain range of $a$ and $b$ depending on $d$ and $\gamma$, where $PM^a$ is the space of pseudomeasures \cite{ferreira}:
\begin{equation}\label{pma}PM^a=\{v\in\mathcal{S}'(\Rd): \hat{v}\in L^1_{\rm loc}(\Rd),\,\text{esssup}\,|\xi|^a|\hat{v}(\xi)|<\infty\}.\end{equation}
Local and global regularity of mild solutions in the scaling hyperdissipative regime $\gamma\ge 1$ for initial data $u_0$ in supercritical Besov spaces and Triebel-Lizorkin spaces are obtained in \cite{baaske}.

The case of particular interest is $d=3$, in which partial regularity is known for the supercritical hyperdissipative regime $1<\gamma<5/4$ \cite{katz}. If $u_0\in H^\delta$ (Sobolev space) for some arbitrary $\delta>0$, then global regularity holds for any $\gamma\in [\frac{5}{4}-\epsilon,\frac{5}{4}]$ where $\epsilon>0$ depends on $\|u_0\|_{H^\delta}$ \cite{colombo}. Finite-time blowup of solutions to a forced system is obtained in \cite{cordoba} for the hypodissipative regime $0\le\gamma<\frac{2}{9}(22-8\sqrt{7})$.

Now let us turn to the statistical model and its associated
statistical criticality. 
\ref{nsedg} has an integral form
\begin{equation}\label{nseint}
u(t)=e^{-t(-\Delta)^\gamma}u_0-\int_0^te^{(s-t)(-\Delta)^\gamma}\mathbf{P}\text{div}(u\otimes u)ds
\end{equation}
where $\mathbf{P}$ denotes the Leray projection onto the divergence-free vector fields. Let $\hat{u}$ be the Fourier transform of $u$ in the spatial variable
\begin{equation}\label{c0}\hat{u}(\xi,t)=\mathcal{F}\{u(\cdot,t)\}(\xi)=c_0\int_{\Rd}u(x,t)e^{-ix\cdot\xi}dx,\ \ \ \text{where}\ c_0=(2\pi)^{-d/2}. \end{equation}
Taking the Fourier transform of \eqref{nseint} and using the
convolution identity $\mathcal{F}\{fg\} = c_0 \mathcal{F}\{f\}*\mathcal{F}\{g\}$ 
(see \cite{folland}*{p.\ 284}), one has
\begin{equation}\label{uhateq1}
\hat{u}(\xi,t)=\hat{u}_0(\xi)e^{-|\xi|^{2\gamma}t}+c_0\int_0^t |\xi|
e^{-|\xi|^{2\gamma}s}\int_{\Rd}\hat{u}(\eta,t-s)\otimes_\xi\hat{u}(\xi-\eta,t-s)d\eta ds.
\end{equation}
Here,
\begin{equation}
a\otimes_\xi b=-i(e_\xi\cdot b)\pi_{\xi^\perp}a
\end{equation}
where $e_\xi=\xi/|\xi|$ is the unit vector in $\xi$-direction and $\pi_{\xi^\perp}$ 
is the perpendicular projection onto the hyperplane orthogonal to $\xi$.
As remarked in \cite{dascaliuc2019jan}, the vector product $\otimes_\xi$ in \eqref{uhateq1} can be replaced by a symmetrized vector product $\odot_\xi$ defined by
\begin{equation}\label{odotproduct}
a\odot_\xi b=\frac{1}{2}(a\otimes_\xi b+b\otimes_\xi a)=-\frac{i}{2}((e_\xi\cdot b)\pi_{\xi^\perp}a + (e_\xi\cdot a)\pi_{\xi^\perp} b).\end{equation}
\begin{remark} If one considers scalar valued functions $\hat{u}$ and replaces
$\odot_\xi$ by ordinary scalar multiplication then the mean flow equation is the (rescaled Fourier transformed) fractional Montgomery-Smith equation mentioned in the introduction; see \autoref{reduction}.
\end{remark}

One can rewrite \eqref{uhateq1} as
\begin{equation}\label{uhateq}
\hat{u}(\xi,t)=\hat{u}_0(\xi)e^{-|\xi|^{2\gamma}t}+c_0\int_0^t |\xi|e^{-|\xi|^{2\gamma}s}\int_{\Rd}\hat{u}(\eta,t-s)\odot_\xi\hat{u}(\xi-\eta,t-s)d\eta ds.
\end{equation} 
Let 
\begin{equation}\label{constantc}
h(\xi)=c_{d,\gamma}|\xi|^{2\gamma-d-1},~~~{\textrm{where}}~~~{{c}_{d,\gamma }}={{\pi }^{-d/2}}\frac{\Gamma (2\gamma -1)\Gamma {{\left( \frac{d+1-2\gamma }{2} \right)}^{2}}}{\Gamma \left( \frac{d+2-4\gamma }{2} \right)\Gamma {{\left( \frac{2\gamma -1}{2} \right)}^{2}}}.\end{equation}
For $\gamma\in(\frac{1}{2},\frac{d+2}{4})$, one can infer from the Fourier transform of radial functions \cite[p.\ 205]{rudin}
\begin{equation}\label{fourier}
\mathcal{F}\{|x|^{-\alpha}\}=2^{\frac{d}{2}-\alpha}\frac{\Gamma(\frac{d-\alpha}{2})}{\Gamma(\frac{\alpha}{2})}|\xi|^{\alpha-d},\ \ \ \forall\, \alpha\in(0,d)
\end{equation}
that 
\begin{equation}\label{hkernel}|\xi|^{2\gamma-1}h(\xi)=h*h(\xi).\end{equation}
The function $h$ defined by \eqref{constantc} is known as a \emph{standard majorizing kernel} \cite{rabi}. The rescaled Fourier transform $\chi(\xi,t)=c_0\frac{\hat{u}(\xi,t)}{h(\xi)}$ satisfies the \emph{normalized Fourier-transformed Navier-Stokes equations}:
\begin{equation}\label{mildnse}
\chi(\xi,t)=\chi_0(\xi)e^{-|\xi|^{2\gamma}t}+\int_0^t |\xi|^{2\gamma}e^{-|\xi|^{2\gamma}s}\int_{\Rd}\chi(\eta,t-s)\otimes_\xi\chi(\xi-\eta,t-s)H(\eta|\xi)d\eta ds
\end{equation}
where $\chi_0(\xi)=\chi(\xi,0)$ and the normalized integral kernel $H$ is given by
\begin{equation}
\label{Hkernel}
H(\eta|\xi)=\frac{h(\eta)h(\xi-\eta)}{|\xi|^{2\gamma-1}h(\xi)}.
\end{equation}
The divergence-free constraint on the initial condition $u_0$ becomes an orthogonality constraint on $\chi_0$:
\begin{equation}\label{divfree}
\xi\cdot\chi_0(\xi)=0\ \ \ \forall\xi\in\R^d\backslash\{0\}.
\end{equation}
At this point, the solution $\chi$ to \eqref{mildnse} can be interpreted probabilistically as the expected value of a ``solution process'' $\mathbf{X}$, defined implicitly
for given initial data $\chi_0$, by the recursive relation
\begin{equation}\label{eq:822191}\X(\xi,t)=\left\{ \begin{array}{*{35}{l}}
   \chi_0(\xi) & \text{if} & {{Y_\root}}\ge t,  \\
   {\X}^{(1)}(W_1, t-{{Y_\root}}){{\odot}_{\xi }}{\X}^{(2)}(W_2, t-{{Y_\root}}) & \text{if} & {{Y_\root}}< t.  \\
\end{array} \right.\end{equation}
Here, $Y_\root$ is an exponentially distributed random variable with mean $|\xi|^{-2\gamma}$; $W_1\in\mathbb{R}^d$ is a random variable independent of $Y_\root$ with probability density $H(\cdot|\xi)$; $W_2=\xi-W_1$; and $\X^{(1)}$ and $\X^{(2)}$ are, conditionally given $Y_\root$, two independent copies of $\X$. 

As usual, we use the Lebesgue measure on $\mathbb{R}^k$ and $\mathbb{C}^k$, $k\in\mathbb{N}$. A subset of a measurable set is said to have full measure if the complement has a zero measure.

\begin{defin}[Solution Process] 
For a given measurable function
$\chi_0:\mathbb{R}^d\to\mathbb{C}^d$, a stochastic process  process
$\mathbf{X}$ satisfying \eqref{eq:822191} in distribution,  is referred to as
a \emph{solution process}. 
If $\mathbb{E}_\xi|{\mathbf{X}}|<\infty$ for all $(\xi,t)\in Q$, 
where $Q\subset \mathbb{R}^d\times[0,\tau)$ is a subset with full measure, then $\mathbf{X}$ is referred to as an
\emph{integrable solution process} on $\mathbb{R}^d\times[0,\tau)$. 
\end{defin}

Applying the Definition \eqref{eq:822191} repeatedly with $\bX$ being replaced by $\bX^{(1)}$ and then by $\bX^{(2)}$, one arrives at a 
{\em doubly stochastic Yule cascade} 
$\bY=\{Y_v\}_{v\in\T}$ indexed by a binary tree 
$\T=\{\root\}\bigcup\left(\cup_{n=1}^\infty\{1,2\}^n\right)$ rooted at $\root$, 
where $Y_\root$ is an exponential random variable with intensity $|\xi|^{2\gamma}$, $Y_1$ 
and $Y_2$ are conditionally 
the exponential random variables with intensities $|W_1|^{2\gamma}$ and 
$|W_2|^{2\gamma}$, and so on down the tree $\T$. A more detailed description is in \autoref{dsy} and \autoref{Section3}. As defined below, the stochastic \emph{explosion time} in which infinitely many branchings of
the DSY occur in finite time, 
and the deterministic \emph{critical time of integrability} at which
the solution process fails to have
an expected value, each play an essential role in the 
construction of a solution process. One may note that the discrete time skeletal tree underlying the continuous time has two  offspring at each generation, rendering it 
statistically supercritical in the language of
branching processes. The following definition concerns other notions of statistical criticality involving the (in)finiteness of the number of branchings in a finite continuous time. 

\begin{defin}\label{expltimeintegtime}
For any vertex $v=(v_1,v_2,...,v_n)\in\T$, denote by $|v|=n$ the genealogical height of $v$, and $v|j=(v_1,...,v_j)$ the truncation up to the $j$'th generation with the convention that $v|0=\root$.
\begin{enumerate}[(i)]
\item The (stochastic) \emph{explosion time} is defined by
\begin{equation}\label{explosiontime}
S=S_{\xi} 
=\underset{n\ge 0 }{\mathop{\sup }}\,\underset{|v|=n}{\mathop{\min }}\,\sum\limits_{j=0}^{n} Y_{v|j}.\end{equation}
\item The (stochastic) {\em hyperexplosion time} is defined by
\begin{equation}\label{hyperexplosiontime}
L=L_{\xi} 
=\underset{n\ge 0 }{\mathop{\sup }}\,\underset{|v|=n}{\mathop{\max }}\,\sum\limits_{j=0}^{n} Y_{v|j}.\end{equation}
\item The (deterministic) \emph{critical time of integrability} is defined by
\begin{equation}\label{tcritical}
t_c=t_c(\chi_0) = \sup\{t\ge 0: \mathbb{E}_\xi[|{\mathbf{X}}(\xi,t)|\1_{[S>t]}]<\infty\text{~~a.e.~~}\xi\in\R^d\}.
\end{equation}
\end{enumerate}
\end{defin}

The existence of a solution process is assured by the construction
of the {\it minimal  solution process} as follows.

\begin{prop} For each measurable
 function $\chi_0:\mathbb{R}^d\to\mathbb{C}^d$, a
 solution process exists. 
\end{prop}
\begin{proof} 
For $t<S$, the recursion \eqref{eq:822191} terminates in finitely many steps, and $\X(\xi,t)$ is uniquely defined as an $\odot$-product of the values of $\chi_0$ evaluated at the vectors where the recursion terminates. See \autoref{cascadeexample} for an example of a closed form of $\bX(\xi,t)$. For $t\ge S$, define 
$\X(\xi,t) = 0$.
\end{proof}

\begin{cor}
If $S=\infty$ almost surely, then the minimal solution process 
is the unique solution process. 
\end{cor}
\begin{proof}
 On the event  $[S=\infty]$, the  recursion \eqref{eq:822191} is a.s.\ finite for any $t\ge0$. Any solution process $\bX(\xi,t)$ can be factored into an 
 ${{\odot}_{\xi }}$-product which coincides that of the minimal solution process. 
 \end{proof}

In the explosive event $[S<\infty]$, non-minimal solution
processes may exist. This is indeed the case for the $\alpha$-Riccati equation, a scalar mean flow equation with multiplicative nonlinearity \cites{alphariccati, transformation2025}, where non-minimal solution processes $\bX(\xi,t)$ are constructed on the hyperexplosion event $[L<t]$.



\begin{prop}\label{819191}
Let $\chi_0:\mathbb{R}^d\to\mathbb{C}^d$ be a measurable function and let $\mathbf{X}$ be a corresponding 
integrable solution process. Then the function $\chi(\xi,t)=\mathbb{E}_\xi{\mathbf{X}}$ is well-defined and measurable on $Q$ with the integrability property 
\begin{equation}\label{intcond}\int_{0}^{t}|\xi |^{2\gamma}{{e}^{-|\xi |^{2\gamma}s}}{\int_{{{\mathbb{R}}^{d}}}{|{{\chi }}(\eta ,s){{\odot}_{\xi }}{{\chi }}(\xi-\eta ,s)|H(\eta |\xi )d\eta ds}<\infty }\ \ \ \forall\,(\xi,t)\in Q.\end{equation}
Moreover, $\chi$ satisfies \eqref{mildnse} everywhere in $Q$. 
\end{prop}
\begin{proof}
The proof is very similar to \cite[Prop.\ 2.2]{dascaliuc2019jan}. Since $\X(\xi,t)\in L^1(\mathbb{P}_\xi)$, the function $\chi(\xi,t)=\int_\Omega \mathbf{X}(\xi,t,\omega)\mathbb{P}_\xi(d\omega)$ is measurable on $Q$ (see \cite[Thm 2.39]{folland}). To show \eqref{intcond}, we take the expected value of the magnitude of both side of \eqref{eq:822191}:
\[\infty>{{\mathbb{E}}_{\xi }}|\X|=|{{\chi }_{0}}(\xi )|{{e}^{-|\xi |^{2\gamma}t}}+\int_{0}^{t}{|\xi |^{2\gamma}{{e}^{-|\xi |^{2\gamma}s}}\int_{{{\mathbb{R}}^{d}}}{{{\mathbb{E}}_{\xi }}\left[\left. |\X^{(1)}{{\odot}_{\xi }}\X^{(2)}|\,\right|\,W_1=\eta, T_0=s\right]H(\eta |\xi )d\eta ds}}\]
where $\X^{(k)}=\X^{(k)}(W_k,t-T)$. Using the conditional independence of $\X^{(1)}$ and $\X^{(2)}$ and the linearity of the expectation, we obtain that for a.e.\ $(\eta,s)\in \mathbb{R}^d\times[0,t)$,
\[
\begin{aligned}
&{{{\mathbb{E}}_{\xi }}\left[\left. |\X^{(1)}{{\odot }_{\xi }}\X^{(2)}|\,\right|\,W_1=\eta, T=s\right]}
\ge \left|{{{\mathbb{E}}_{\xi }}\left[\left. \X^{(1)}{{\odot }_{\xi }}\X^{(2)}\,\right|\,W_1=\eta, T=s\right]}\right|
\\
&=
\left|{{\mathbb{E}}_{{\eta}}}{{\X}^{(1)}}(\eta,t-s){\odot}_{\xi} {{\mathbb{E}}_{{\xi-\eta}}}{{\X}^{(2)}}(\xi-\eta,s)\right|
=|\chi(\eta,t-s) {{\odot }_{\xi }}\chi(\xi-\eta,t-s) |,
\end{aligned}\]
and thus (\ref{intcond}) follows. To show $\chi=\mathbb{E}_\xi \X$ satisfies \eqref{mildnse}, we take expectation of both sides of \eqref{eq:822191} (with $(\xi,t)\in Q$):
\[
\chi(\xi,t)={{\mathbb{E}}_{\xi }}[\X(\xi,t)]={{\chi }_{0}}(\xi ){{e}^{-|\xi |^{2\gamma}t}}+\int_{0}^{t}{|\xi |^{2\gamma}{{e}^{-|\xi |^{2\gamma}s}}\int_{{{\mathbb{R}}^{d}}}{{{\mathbb{E}}_{\xi }}\left[\left. \X^{(1)}{{\odot }_{\xi }}\X^{(2)}\,\right|\,W_1=\eta, T=s\right]H(\eta |\xi )d\eta ds}}
\]
Thus, \eqref{mildnse} holds in $Q$. 
\end{proof}
\begin{remark}
While the present paper mainly focuses on integrable minimal solution process, there is a possibility of justifying the "expected value" of a non-integrable solution process, e.g.\ via Cauchy principle value integral, which can still establish a deterministic solution to \eqref{mildnse}.
\end{remark}

Special geometry on the two-dimensional plane leads to a more specific representation of the $\odot_\xi$-product as follows.

\begin{prop}\label{cased=2}
Let $f,g:\R^2\to\C^2$ be vector fields that satisfy the orthogonality condition: 
\[\xi\cdot f(\xi)=\xi\cdot g(\xi)=0\ \ \ \forall\xi\in\R^2\backslash\{0\}.\]
For $\xi,\eta,\zeta\in\R^2$, $\xi=\eta+\zeta$, one has
\begin{equation}\label{odot}f(\eta)\odot_\xi g(\zeta)=\frac{i}{2}\textup{sign}(\xi\times\eta)(f(\eta)\cdot e_{\eta^\perp})(g(\zeta)\cdot e_{\zeta^\perp})\sin(\theta_{\xi,\eta}-\theta_{\xi,\zeta})e_{\xi^\perp}.\end{equation}
Here, $a\times b=a^\perp\cdot b=a_1b_2-a_2b_1$ if $a=(a_1,a_2)$ and $b=(b_1,b_2)$; $\theta_{u,v}\in[0,\pi]$ is the angle between two vectors $u$ and $v$; $\xi^\perp$ denotes the rotation by $+90^\circ$ of $\xi$, i.e.\ $\xi^\perp=\left[ \begin{matrix}
   0 & -1  \\
   1 & 0  \\
\end{matrix} \right]\xi$.
\end{prop}
\begin{proof}
Let $\tilde{f}(\xi)=f(\xi)\cdot e_{\xi^\perp}$ and $\tilde{g}(\xi)=g(\xi)\cdot e_{\xi^\perp}$. By the bilinearity of the $\odot_\xi$-product,
\[f(\eta)\odot_\xi g(\zeta)=\tilde{f}(\eta)\tilde{g}(\zeta)e_{\eta^\perp}\odot_\xi e_{\zeta^\perp}.\]
By the definition of the $\odot$-product,
\begin{eqnarray*}
e_{\eta^\perp}\odot_\xi e_{\zeta^\perp}&=&\frac{1}{2}(e_{\eta^\perp}\otimes_\xi e_{\zeta^\perp}+e_{\zeta^\perp}\otimes_\xi e_{\eta^\perp})\\
&=&-\frac{i}{2}\left((e_\xi\cdot e_{\eta^\perp})\pi_{\xi^\perp} e_{\zeta^\perp}+(e_\xi\cdot e_{\zeta^\perp})\pi_{\xi^\perp} e_{\eta^\perp}\right)\\
&=&-\frac{i}{2}\left(\cos\theta_{\xi,\eta^\perp}\cos\theta_{\xi^\perp,\zeta^\perp}+\cos\theta_{\xi,\zeta^\perp}\cos\theta_{\xi^\perp,\eta^\perp}\right)e_{\xi^\perp}\\
&=&-\frac{i}{2}\left(\cos\theta_{\xi,\eta^\perp}\cos\theta_{\xi,\zeta}+\cos\theta_{\xi,\zeta^\perp}\cos\theta_{\xi,\eta}\right)e_{\xi^\perp}.
\end{eqnarray*}
\begin{figure}[h!]
\centering\includegraphics[scale=.65]{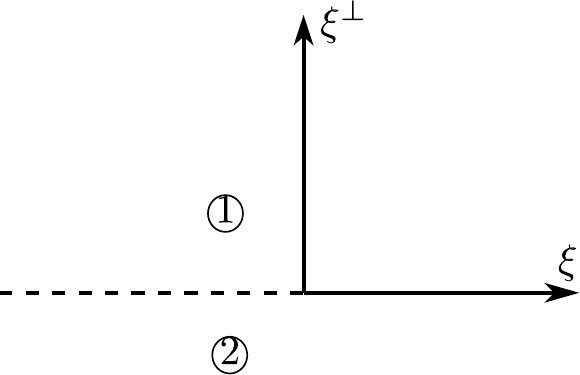}
\caption{Two halves of the plane}
\label{planesplit}
\end{figure}
Based at the origin, the vector $\xi$ partitions the plane into two halves, labeled as $\textcircled{1}$ and $\textcircled{2}$ in \autoref{planesplit}.
\begin{eqnarray*} 
\textcircled{1}&=&\{\eta\in\mathbb{R}^2:\ \xi\times\eta=\xi^\perp\cdot\eta>0\}\\
\textcircled{2}&=&\{\eta\in\mathbb{R}^2:\ \xi\times\eta=\xi^\perp\cdot\eta<0\}
\end{eqnarray*}
If $\eta\in\textcircled{1}$ then $\xi\times\eta>0$. We have $\cos\theta_{\xi,\eta^\perp}=-\sin\theta_{\xi,\eta}$ and $\cos\theta_{\xi,\zeta^\perp}=\sin\theta_{\xi,\zeta}$. Thus
\[e_{\eta^\perp}\odot_\xi e_{\zeta^\perp}=-\frac{i}{2}\sin(\theta_{\xi,\zeta}-\theta_{\xi,\eta})e_{\xi^\perp}=\frac{i}{2}\sin(\theta_{\xi,\eta}-\theta_{\xi,\zeta})e_{\xi^\perp}.\]
If $\eta\in\textcircled{2}$ then $\xi\times\eta<0$. We have $\cos\theta_{\xi,\eta^\perp}=\sin\theta_{\xi,\eta}$ and $\cos\theta_{\xi,\zeta^\perp}=-\sin\theta_{\xi,\zeta}$. Thus
\[e_{\eta^\perp}\odot_\xi e_{\zeta^\perp}=-\frac{i}{2}\sin(\theta_{\xi,\eta}-\theta_{\xi,\zeta})e_{\xi^\perp}.\]
\end{proof}
For \ref{nsedg} to be nontrivial, it is necessary that $d\ge 2$. Upon removal of the pressure and the divergence-free constraint, one arrives at the \emph{fractional Burgers equation}, a well-known toy model for the Navier-Stokes equations:
\begin{equation}\label{burgerseq}
u_t+(-\Delta)^\gamma u+\frac{1}{2}(u^2)_x=0,\ \ u(\cdot,0)=u_0,
\end{equation}
where $u=u(x,t)$, $x\in\R$, $t\ge 0$. For the corresponding range of values of $\gamma$, namely $\gamma\in (\frac{1}{2}, \frac{3}{4})$, one can also interpret the function $\chi(\xi,t)=\frac{c_0\hat{u}(\xi,t)}{h(\xi)}$ probabilistically as the expected value of a solution process $\mathbf{X}$ defined in a similar manner as \eqref{eq:822191}:
\begin{equation}\label{burgersprocess}\X(\xi,t)=\left\{ \begin{array}{*{35}{l}}
   \chi_0(\xi) & \text{if} & {{Y_\root}}\ge t,  \\
   -\frac{i}{2}\,\sign(\xi){\X}^{(1)}(W_1, t-{{Y_\root}}){\X}^{(2)}(W_2, t-{{Y_\root}}) & \text{if} & {{Y_\root}}< t.  \\
\end{array} \right.\end{equation}
The fractional Navier-Stokes equations ($d\ge 2$) and the fractional Burgers equation ($d=1$) have essentially the same underlying stochastic structure called a doubly stochastic Yule cascade to be defined in the next section.
\section{$(d,\gamma)$-DSY cascade} 
\label{Section3}
Let $\T=\{\root\}\bigcup\left(\cup_{n=1}^\infty\{1,2\}^n\right)$ be the full binary tree rooted at $\root$. The notion of \emph{doubly stochastic Yule (DSY) cascades} was introduced in \cite{part1_2021} as a tree-indexed family of random variables $\{Y_v=\lambda^{-1}_vT_v\}_{v\in\T}$, where $\{\lambda_v\}_{v\in\T}$ is a family of positive random variables independent of the family of  i.i.d.\ mean-one exponentially distributed  random variables
$\{T_v\}_{v\in\T}$. We define below a special subclass of
DSY cascades corresponding to a special choice of
transition probability kernel $H$ and stochastic
intensities $\{\lambda_v\}_{v\in\T}$.
\begin{defin}\label{dsy}
For $d\ge 1$ and $\gamma\in(\frac{1}{2},\frac{d+2}{4})$, a $(d,\gamma)$-DSY cascade is a tree-index family $\{Y_v=\lambda^{-1}_vT_v\}_{v\in\T}$ in which $\{T_v\}_{v\in\T}$ is an i.i.d.\ mean-one exponentially distributed family of random variables and $\lambda_v=|W_v|^{2\gamma}$, where $\{W_v\}_{v\in\T}$ is a family of random vectors in $\R^d$ distributed as follows:
\begin{enumerate}[(i)]
\item $W_\root=\xi,$\vspace{-4pt}
\item Given
$\sigma(W_u:u\in\T, |u|\le|v|)$, the conditional distribution of $W_{v1}$ is $H(\cdot|W_v)$ given by \eqref{Hkernel}, and $W_{v2}=W_v-W_{v1}$,\vspace{-4pt}
\item Given $W_{v1}$ and $W_{v2}$, the sub-families $\{W_{v1\sigma}\}_{\sigma\in\mathbb{T}}$ and $\{W_{v2\sigma}\}_{\sigma\in\mathbb{T}}$ are independent of each other.
\end{enumerate}
\end{defin}

\begin{remark}
Because of the invariance of $H(\cdot|\xi)$ in \eqref{Hkernel} with respect to the transformation $\eta\mapsto \xi-\eta$, the roles of $W_{v1}$ and $W_{v2}$ are interchangeable in \autoref{dsy}.
\end{remark}
The DSY cascade associated with the fractional Navier-Stokes equations arises when recursively applying formula \eqref{eq:822191}. \autoref{cascade} is a cascade figure that illustrates the families $\{{Y}_{v}\}_{v\in\T}$ and $\{W_{v}\}_{v\in\T}$ induced from the recursion \eqref{eq:822191}. Although the definition of the $\odot_\xi$-product requires $d\ge 2$, the DSY cascade can be defined for any dimension $d\ge 1$.
\begin{figure}[h!]
\centering
\includegraphics[scale=1]{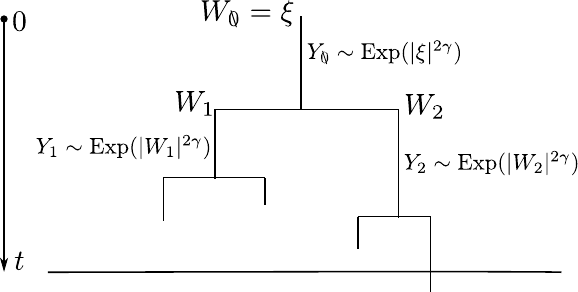}
\caption{Cascade figure that illustrates $\{{Y}_{v}\}_{v\in\T}$ and $\{W_{v}\}_{v\in\T}$.}
\label{cascade}
\end{figure}
With the time evolution of the $(d,\gamma)$-cascade $\{Y_v\}_{v\in\T}$, the recursion \eqref{eq:822191} terminates at a vertex $v\in\T$ if and only if
 \[
 \sum_{j=0}^{|v|-1}Y_{v|j}<t\le\sum_{j=0}^{|v|}Y_{v|j}\,.
 \]
Such a vertex will be referred to as a \emph{$t$-leaf}. The continuous-parameter Markov process of sets of $t$-leaves is denoted by
\begin{equation}
\label{tleaves}
\partial{V}(\xi,t) =\left\{v\in\T: \sum_{j=0}^{|v|-1}Y_{v|j}<t\le\sum_{j=0}^{|v|}Y_{v|j}\right\}.
 \end{equation}
We denote the set of $t$-\emph{internal vertices} by
 \begin{equation}
\label{tancestors}
\Vo(\xi,t) = \left\{u\in\T:\ \sum_{j=0}^{|u|}Y_{u|j}<t\right\},
\end{equation}
and the $t$-subtree associated with $\{Y_v\}_{v\in\T}$ by
\begin{equation}\label{ttree}
V(\xi,t) = \Vo(\xi,t) \cup \partial{V(\xi,t)}.
\end{equation}
\begin{defin}
Let $S=S_\xi$ and $L=L_\xi$ be the explosion and hyperexplosion times defined in \eqref{explosiontime}-\eqref{hyperexplosiontime}. We refer to the events
\begin{itemize}
\item $[S\le t]=[\Vo(t)\text{ is an infinite set}]$ as the event of \emph{explosion by time $t$},
\item $[S> t]=[\Vo(t)\text{ is a finite set}]$ as the event of \emph{non-explosion by time $t$},
\item $[L\le t]=[\partial V(t)=\emptyset]$ as the event of \emph{hyperexplosion by time $t$},
\item $[L> t]$ as the event of \emph{non-hyperexplosion by time $t$}.
\end{itemize}
The DSY cascade is said to be \emph{explosive} if $\P(S<\infty)>0$, and \emph{hyperexplosive} if $\P(L<\infty)>0$. Otherwise, it is said to be \emph{non-explosive} and \emph{non-hyperexplosive}, respectively.
\end{defin}
\noindent Whenever $V(\xi,t)$ is a finite set, i.e. on the event $[S>t]$, one has 
\begin{equation}\label{leafancestorcounts}
|\partial V(\xi,t)| = |\Vo(\xi,t)|+1,
\quad |V(\xi,t)| = 2 |\Vo(\xi,t)|+1.
\end{equation}
Moreover, the solution process $\mathbf{X}(\xi,t)$ at \eqref{eq:822191} is a $\odot_\xi$-product, in a suitable order, of the initial data $\chi_0$ evaluated on the $t$-leaves $\partial{V(\xi,t)}$. For example, in the tree realization in \autoref{cascadeexample},
\[\partial V(\xi,t)=\{11,121,122,21,22\},\]
\[\X(\xi,t)=(\chi_0(W_{11})\odot_{W_1}(\chi_0(W_{121})\odot_{W_{12}}\chi_0(W_{122})))\odot_\xi(\chi_0(W_{21})\odot_{W_2}\chi_0(W_{22})).\]
\begin{figure}[h!]
\centering
\includegraphics[scale=1]{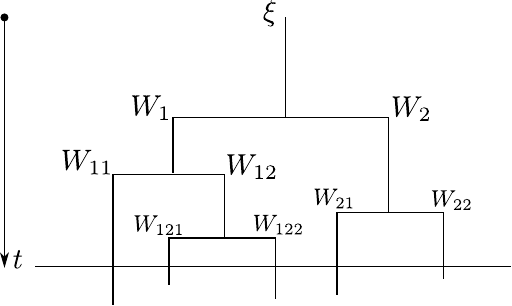}
\caption{Cascade that is non-explosive by time $t$.}
\label{cascadeexample}
\end{figure}

The purpose of the rest of this section is to give criteria for explosion, non-explosion, hyperexplosion, and non-hyperexplosion in terms of the dimension $d$ and the fractional power $\gamma$ (\autoref{exp-nonexp}).
Our criteria involve the distribution of the ratio of the magnitude of the offspring relative to that of the parent.  The next subsection presents some of these results.

\subsection{Distribution of the ratios}
In the following proposition, $\overleftarrow{v}$ denotes the parent vertex of $v$, and $\partial\T=\{1,2\}^{\mathbb{N}}$ denotes the set of rays emanating from the root.
\begin{prop}\label{ratiodist}
Let $d\ge 1$ and $\gamma\in(\frac{1}{2},\frac{d+2}{4})$. For any $v\in\T\backslash\{\root\}$, the pdf of the ratio $R_v=\frac{|W_v|}{|W_{\overleftarrow{v}}|}$ is given by
\begin{itemize}
\item For $d=1$,
\begin{equation}\label{ratiopdf1}
g_{1,\gamma}(r)=c_{1,\gamma} (|r-r^2|^{2\gamma-2}+(r+r^2)^{2\gamma-2})\end{equation}
where $c_{1,\gamma}>0$ is the constant given by \eqref{constantc}.
\item For $d>1$,
\begin{equation}\label{ratiopdf}
g_{d,\gamma}(r)={{C}_{d,\gamma }} \frac{r^{2\gamma-2}}{{{\left( {{r}^{2}}+1 \right)}^{a}}}\ {{}_{2}}{{F}_{1}}\left( \frac{a}{2},\frac{a+1}{2};\frac{d}{2};\frac{4{{r}^{2}}}{{{\left( {{r}^{2}}+1 \right)}^{2}}} \right)
\end{equation}
where $a=\frac{d+1}{2}-\gamma>0$ and
\[{{{C}}_{d,\gamma }}=\frac{2\Gamma (2\gamma -1)\Gamma {{(a)}^{2}}}{\Gamma \left( \frac{d}{2} \right)\Gamma \left( 2a-\frac{d}{2} \right)\Gamma {{\left( \frac{2\gamma -1}{2} \right)}^{2}}}.\]
\end{itemize}
Moreover, for any $s\in\partial\mathbb{T}$, the sequence
$\{R_{s|k}\}_{k\ge 1}$, is i.i.d.
\end{prop}
\begin{proof}
Let us first consider $v=1$. Let $R=R_1=|W_1|/|\xi|$ and let $k:[0,\infty)\to\R$ be any bounded measurable function. For $d=1$,
\begin{eqnarray*}
\mathbb{E}[ k(R) ]&=&\int_{-\infty }^{\infty }{k \left( \frac{|\eta |}{|\xi |} \right)H(\eta |\xi )d\eta }={{c}_{1,\gamma }}\int_{-\infty }^{\infty }{k \left( \frac{|\eta |}{|\xi |} \right)\frac{|\eta {{|}^{2\gamma -2}}|\xi -\eta {{|}^{2\gamma -2}}}{|\xi {{|}^{4\gamma -3}}}d\eta }\\
&=&{{c}_{1,\gamma }}\int_{0}^{\infty }{k \left( \frac{\eta }{|\xi |} \right)\frac{{{\eta }^{2\gamma -2}}}{|\xi {{|}^{4\gamma -3}}}\left( |\xi -\eta {{|}^{2\gamma -2}}+|\xi +\eta {{|}^{2\gamma -2}} \right)d\eta }\\
&=&{{c}_{1,\gamma }}\int_{0}^{\infty }{k \left( r \right){{r}^{2\gamma -2}}\left( |1-r{{|}^{2\gamma -2}}+(1+r)^{2\gamma -2} \right)dr}.
\end{eqnarray*}
This gives the probability density function \eqref{ratiopdf1}. It is easy to check that the inverse ratio $\frac{1}{R}$ has a probability density function
\[\tilde{g}_{1,\gamma}(r)={{g}_{1,\gamma}}\left( \frac{1}{r} \right)\frac{1}{{{r}^{2}}}=\frac{{{g}_{1,\gamma}}(r)}{{{r}^{6\gamma -4}}}.\]
Thus, $R$ is symmetrically distributed on the multiplicative group $(0,\infty)$ if and only if $\gamma=\frac{2}{3}$. For $d\ge 2$,
\[\mathbb{E}[k(R)]=\int_{{{\mathbb{R}}^{d}}\times {{\mathbb{R}}^{d}}}{k\left( \frac{|\eta|}{|\xi |} \right)H(\eta|\xi)}d\eta={{c}_{d,\gamma }}\int_{{{\mathbb{R}}^{d}}}{k\left( \frac{|\eta|}{|\xi |} \right)\frac{|\xi {{|}^{d+2-4\gamma }}}{|\eta{{|}^{d+1-2\gamma }}|\xi-\eta{{|}^{d+1-2\gamma }}}d\eta}\]
Let $u=\xi/|\xi|$ and $\tilde{u}=\eta/|\xi|$. Then one can simplify $\EXP[k(R)]$ as
\begin{equation}\label{27241}\mathbb{E}[k(R)]={{c}_{d,\gamma }}\int_{{{\mathbb{R}}^{d}}}{k(|\tilde{u}|)\frac{1}{|\tilde{u}{{|}^{d+1-2\gamma }}|u-\tilde{u}{{|}^{d+1-2\gamma }}}d\tilde{u}}.\end{equation}
Choose the spherical coordinates $(r,\phi,w)$ in $\mathbb{R}^d$ such that $\tilde{u}=u r\cos\phi+wr\sin\phi$ where $r=|\tilde{u}|$, $w\in \mathbb{S}_u^{d-2}$, the unit sphere in the hyperplane perpendicular to $u$, and $\phi\in[0,\pi]$ is the angle between $u$ and $\tilde{u}$. Then $d\tilde{u}=r^{d-1}\sin^{d-2}\phi\, dr\,d\phi\,dw$. 
Therefore, 
\begin{eqnarray}\label{EkR}\mathbb{E}[k(R)]&=&\bar{C}_{d,\gamma}\int_{0}^{\infty }{\int_{0}^{\pi }{k(r){{r}^{2\gamma-2 }}\frac{\sin^{d-2}(\phi)}{{{(1-2r\cos \phi +{{r}^{2}})}^{\frac{d+1-2\gamma}{2} }}}d\phi dr}}
\end{eqnarray}
where
\begin{equation}\label{constantC}{\bar{C}_{d,\gamma }}=|{{\mathbb{S}}^{d-2}}|{{c}_{d,\gamma }}=\frac{2}{\sqrt{\pi }}\frac{\Gamma (2\gamma -1)\Gamma {{\left( a \right)}^{2}}}{\Gamma \left( \frac{d-1}{2} \right)\Gamma \left( 2a-\frac{d}{2} \right)\Gamma {{\left( \frac{2\gamma -1}{2} \right)}^{2}}}.\end{equation}
Then the p.d.f.\ of $R$ is
\[g_{d,\gamma}(r)=\bar{C}_{d,\gamma}{{r}^{2\gamma-2 }}\int_{0}^{\pi /2}{{{\sin }^{d-2}}\phi \left( {{(1-2r\cos \phi +{{r}^{2}})}^{-a}}+{{(1+2r\cos \phi +{{r}^{2}})}^{-a}} \right)d\phi }\]
With the aid of Mathematica to evaluate the above integral (or by using the double-angle identity $\sin\phi=2\sin(\phi/2)\cos(\phi/2)$ and \cite[Formula 3.682, p.415]{gradshteyn} to compute the inner integral of \eqref{EkR}), one obtains
\begin{eqnarray*}
g_{d,\gamma}(r)&=&\sqrt{\pi}{\bar{C}_{d,\gamma }} \frac{\Gamma\left(\frac{d-1}{2}\right)}{\Gamma\left(\frac{d}{2}\right)}\frac{r^{2\gamma-2}}{{{\left( {{r}^{2}}+1 \right)}^{a}}}\ {{}_{2}}{{F}_{1}}\left( \frac{a}{2},\frac{a+1}{2};\frac{d}{2};\frac{4{{r}^{2}}}{{{\left( {{r}^{2}}+1 \right)}^{2}}} \right)\\
&=&{C}_{d,\gamma}\frac{r^{2\gamma-2}}{{{\left( {{r}^{2}}+1 \right)}^{a}}}\ {{}_{2}}{{F}_{1}}\left( \frac{a}{2},\frac{a+1}{2};\frac{d}{2};\frac{4{{r}^{2}}}{{{\left( {{r}^{2}}+1 \right)}^{2}}} \right).
\end{eqnarray*}
By exactly the same calculation and by conditioning on $W_v$, one can show that the pdf of $R_v$ is $g_{d,\gamma}$ for any $v\in\T\backslash\{\root\}$. For each fixed $s\in\partial\T$, the independence of the sequence $\{R_{s|k}\}_{k\ge 1}$ follows from induction and
the Markov property of the sequence $\{W_{s|k}\}_{k\ge 0}$.
\end{proof}
\begin{defin}(see \cite[Def.\ 2.5, Lem.\ 7.1]{part2_2021})
The DSY cascade $\{Y_v\}_{v\in\T}$ is said to be \emph{self-similar} if for any initial state $W_\root=\xi$ and for any $s\in\partial\T$, the ratios $\{R_{s|j}=|W_{s|j}|/|W_{s|j-1}|\}_{j\ge 1}$ are i.i.d.\ with a common distribution independent of $s$.
\end{defin}
\begin{prop}\label{symmetry1}
Let $d\ge 1$ and $\gamma\in(\frac{1}{2},\frac{d+2}{4})$. We have the following statements.
\begin{enumerate}[(a)]
\item A $(d,\gamma)$-DSY cascade is self-similar. 
\item The ratio $R_v=\frac{|W_v|}{|W_{\overleftarrow{v}}|}$ is symmetrically distributed on the multiplicative group $(0,\infty)$, i.e.\ having the same distribution as $R_v^{-1}$, if and only if $\gamma=\frac{d+3}{6}$.
\item $\EXP[\ln R_v]$ has the same sign as $\gamma-(d+3)/6$. That is,
\[\EXP[\ln {{R}_{v}}]\,\left\{ \begin{array}{*{35}{l}}
   <0 & \textup{if} & \gamma <(d+3)/6,  \\
   =0 & \textup{if} & \gamma =(d+3)/6,  \\
   >0 & \textup{if} & \gamma >(d+3)/6.  \\
\end{array} \right.\]
\end{enumerate}
\end{prop}
\begin{proof}
By \autoref{ratiodist}, the ratios $\{R_{s|k}\}_{k\ge 1}$ are i.i.d. Thus, a $(d,\gamma)$-DSY cascade is self-similar. The inverse ratio $\frac{1}{R}$ has a probability density function
\[\tilde{g}_{d,\gamma}(r)={{g}_{d,\gamma}}\left( \frac{1}{r} \right)\frac{1}{{{r}^{2}}}=\frac{{{g}_{d,\gamma}}(r)}{{{r}^{4\gamma -2a-2}}}=\frac{g_{d,\gamma}(r)}{r^{6\gamma-d-3}}.\]
Therefore, $R$ is symmetrically distributed on the multiplicative group $(0,\infty)$ if and only if $\gamma=\frac{d+3}{6}$. To prove Part (c), we decompose $\EXP[\ln R_v]$ into positive and negative parts as follows.
\begin{equation}\label{Rvsplit}
\EXP[\ln R_v]=\int_0^\infty g_{d,\gamma}(r)\ln r\,dr=\int_0^1\ldots+\int_1^\infty\ldots
\end{equation}
Using the change of variable $r\mapsto 1/r$ for the first integral of RHS\eqref{Rvsplit}, one can rewrite \eqref{Rvsplit} as
\begin{align}
\label{Rvsplit2}\EXP[\ln R_v]&=\int_1^\infty \tilde{g}_{d,\gamma}(r)\ln(1/r)\,dr+\int_1^\infty g_{d,\gamma}(r)\ln r\,dr=\int_1^\infty (1-r^{d+3-6\gamma})g_{d,\gamma}(r)\ln r\,dr.
\end{align}
Part (c) follows straightforwardly from this representation of $\EXP[\ln R_v]$.
\end{proof}

\begin{prop}\label{angle-distr-prop}
Let $d>1$. For each $v\in\T$, let $\Phi_{v1}\in[0,\pi]$ be the angle between $W_v$ and $W_{v1}$, and $\Phi_{v2}\in[0,\pi]$ be the angle between $W_v$ and $W_{v2}$. The joint distribution density of the random pair $(\Phi_1,\Phi_2)\in\triangle:=\{(\phi_1,\phi_2):\ \phi_1,\phi_2\ge 0,\ \phi_1+\phi_2\le\pi\}$ is given by
\begin{equation}\label{angle-distr}
f_{d,\gamma}(\phi_1,\phi_2)=\bar{C}_{d,\gamma} (\sin\phi_1\sin\phi_2)^{2\gamma-2}{\sin }^{d+1-4\gamma}(\phi_1+\phi_2)
\end{equation}
where the constant $\bar{C}_{d,\gamma}$ is given by \eqref{constantC}.
\end{prop}
\begin{proof}
It is sufficient to consider the root vertex $v=\root$. Let us denote the values of random variables $W_\root,W_1,\Phi_1,\Phi_2$ respectively by $\xi, \eta,\phi_1,\phi_2$. Let $u=\xi/|\xi|$ and $\tilde{u}=\eta/|\xi|$. By the law of sines in the triangle made by $u$, $\tilde{u}$, $u-\tilde{u}$, we get 
\begin{eqnarray}
\label{210241}r=|\tilde{u}|&=&\frac{\sin \phi_2}{\sin(\pi-\phi_1-\phi_2)}=\frac{\sin\phi_2}{\sin(\phi_1+\phi_2)}\\
\label{28241}|u-\tilde{u}|&=&\frac{\sin \phi_1}{\sin(\pi-\phi_1-\phi_2)}=\frac{\sin\phi_1}{\sin(\phi_1+\phi_2)}.
\end{eqnarray}
Thus, the Jacobian of the change of variables from $(r,\phi,w)$ to $(\phi_1,\phi_2,w)$ is $\frac{\partial(r,\phi,w)}{\partial(\phi_1,\phi_2,w)}=\frac{-\sin\phi_1}{\sin^2(\phi_1+\phi_2)}$. We obtain
\begin{equation}\label{27242}
d\tilde{u}=r^{d-1}\sin^{d-2}\phi\, dr d\phi dw=\frac{(\sin\phi_1\sin\phi_2)^{d-1}}{\sin^{d+1}(\phi_1+\phi_2)}d\phi_1d\phi_2dw.
\end{equation}
Let $k:\triangle\to\R$ be a continuous and bounded function. By \eqref{27241} and \eqref{27242},
\begin{eqnarray*}
\mathbb{E}[k(\Phi_1,\Phi_2)]&=&{{c}_{d,\gamma }}\int_{{{\mathbb{R}}^{d}}}{k(\phi_1,\phi_2)\frac{1}{|\tilde{u}{{|}^{d+1-2\gamma }}|u-\tilde{u}{{|}^{d+1-2\gamma }}}d\tilde{u}}\\
&=&{{c}_{d,\gamma }}\int_{\mathbb{S}_{u}^{d-2}}{\int_{\triangle }{k({{\phi }_{1}},{{\phi }_{2}}){{(\sin {{\phi }_{1}}\sin {{\phi }_{2}})}^{2\gamma -2}}{{\sin }^{d+1-4\gamma }}({{\phi }_{1}}+{{\phi }_{2}})d{{\phi }_{1}}d{{\phi }_{2}}dw}}\\
&=&{{{\bar{C}}}_{d,\gamma }}\int_{\triangle }{k({{\phi }_{1}},{{\phi }_{2}}){{(\sin {{\phi }_{1}}\sin {{\phi }_{2}})}^{2\gamma -2}}{{\sin }^{d+1-4\gamma }}({{\phi }_{1}}+{{\phi }_{2}})d{{\phi }_{1}}d{{\phi }_{2}}}.
\end{eqnarray*}
Hence, the joint distribution of $(\Phi_2,\Phi_2)$ is given by \eqref{angle-distr}.
\end{proof}
\begin{remark} In the classical case $d=3$ and $\gamma=1$, Le Jan and Sznitman \cite[Prop.\ 1.2]{lejan} remarked that the angular vector $(\Phi_1,\Phi_2)$ is uniformly distributed in the triangle $\triangle$. It is easy to see that this is the case if and only if  $d=3$ and $\gamma=1$. 
\end{remark}

\begin{remark}
As far as the distribution of the ratios $R_v$ or of the angles $(\Phi_1,\Phi_2)$ is concerned, one can assume $|\xi|=1$. The symmetry with respect the multiplicative group $(0,\infty)$ of the ratio $R$ 
as noted in \autoref{symmetry1} can be seen as a consequence of the invariance of the joint
distribution of $(\Phi_1,\Phi_2)$ with respect to the Kelvin inversion $W_1\to W^*_1$ with respect to the unit sphere. This can be proved as follows. In terms of the angles, the Kelvin inversion is the transformation (see \autoref{Kelvin})
\begin{equation}\label{coord-transf}
(\phi_1,\phi_2) \rightarrow (\phi_1, \phi_3), \qquad \phi_3=\pi-(\phi_1+\phi_2).
\end{equation}
Since $\sin(\phi_1+\phi_2) = \sin(\phi_3)$, the joint density of the angles is 
invariant iff 
\[
d+1-4\gamma= 2\gamma -2
\]
i.e. iff $\gamma=\tfrac{d+3}{6}$. 
\end{remark}
\begin{figure}[h!]
    \centering
    \includegraphics[scale=.8]{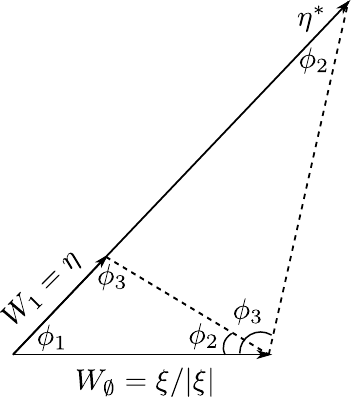}
    \caption{Kelvin inversion $\eta\mapsto\eta^*$, where $W_1=\eta$ and $\eta^*=\frac{1}{|\eta|^2}\eta$.}
    \label{Kelvin}
\end{figure}
\begin{remark}
Alternatively, instead of angles $\Phi_1$ and $\Phi_2$, one can describe the basic triangle formed by $W_v$, $W_{v1}$, and $W_{v2}$ in terms of normalized side lengths: $1=|W_v|/|W_v|$, $R_1=|W_{v1}|/|W_v|$, and  $R_2=|W_{v2}|/|W_v|$. The  joint distribution of $R_1$ and $R_2$ is given by the pdf
\begin{equation}\label{r1r2-distr}
h_{d,\gamma}(r_1,r_2)=\tilde{C}_{d,\gamma}\frac{A(1,r_1,r_2)^{d-3}}{(r_1r_2)^{d-2\gamma}} \IND{D}(r_1, r_2),
\end{equation}
where $$A(1,r_1,r_2)=\frac{1}{4}\sqrt{(1+r_1+r_2)(1+r_1-r_2)(1+r_2-r_1)(r_1+r_2-1)}$$ is the area of the triangle with side lengths $1$, $r_1$, $r_2$ (Heron's formula) and $$
D=\{ (r_1, r_2)\in\mathbb{R}^2 : r_1, r_2>0, r_1+r_2>1, |r_1-r_2|<1\}.
$$
\end{remark}
The moments of random variables $R_v$ will play a defining role in our explosive/non-explosive criteria of the corresponding DSY cascade.

\begin{prop}\label{moments}
Let $d\ge 1$ and $\gamma\in(\frac{1}{2},\frac{d+2}{4})$. One has:
\begin{enumerate}[(i)] 
\item For all $b\in\mathbb{R}$, $\EXP[R^b]=\EXP[R^{\bar{b}}]$ where $\bar{b}=d+3-6\gamma-b$.
\item Let $b^*=\frac{d+3}{2}-3\gamma$. Then $\min\limits_{b\in\R}\EXP[R^b]=\EXP[R^{b^*}]$ and
\begin{align}
\label{moment}\EXP[R^{b^*}]&=\left\{ \begin{array}{*{35}{r}}
   \frac{2\,\Gamma(1-\gamma)^3\,\Gamma(2\gamma-1)^2}{\sqrt{\pi}\,\Gamma\!\left(\frac{3}{2}-2\gamma\right)\,\Gamma\!\left(\gamma-\frac{1}{2}\right)^2\,\Gamma(\gamma)} & \textup{if} & d=1,  \\[3ex]
{\bar{C}_{d,\gamma}}\int_{0}^{\pi}\int_0^{\pi-\phi_1}(\sin\phi_2\,\sin(\phi_1+\phi_2))^{\frac{d-1}{2}-\gamma}\sin^{2\gamma-2}(\phi_1)d\phi_2d\phi_1 & \textup{if} & d>1.  \\
\end{array} \right.\end{align}
where $\bar{C}_{d,\gamma}$ is the constant given by \eqref{constantC}.
\item $\EXP[R^b]\ge 1$ for all $(b,\gamma)\in(0,\infty)\times[\frac{d+3}{6},\frac{d+2}{4})\bigcup\,(-\infty,0)\times (\frac{1}{2},\frac{d+3}{6}]$.
\end{enumerate}
\end{prop}
\begin{proof}
(i) Consider the case $d=1$. We know that $\EXP[R^b]=\int_0^\infty g_{1,\gamma}(r)r^bdr$, where $g_{1,\gamma}$ the pdf of $R_1$ given by \eqref{ratiopdf1}. Using the change of variable $r\mapsto 1/r$, we get
\[\EXP[R^b]=\int_0^\infty g_{1,\gamma}(r)r^{d+3-6\gamma-b}dr=\EXP[R^{\bar{b}}].\]
Consider the case $d\ge 2$. Let $a=\frac{d+1}{2}-\gamma>0$ and 
\[\tilde{F}_{d,\gamma}(r)=\,_2F_1\left(\frac{a}{2},\frac{a+1}{2},\frac{d}{2},\frac{4r^2}{(r^2+1)^2}\right).\]
For any $b\in\R$,
\begin{equation}\label{26241}\EXP[R^b]=\int_0^\infty C_{d,\gamma}\frac{r^{2\gamma-2+b}}{(r^2+1)^a}\tilde{F}_{d,\gamma}(r)dr.\end{equation}
Note that $\tilde{F}_{d,\gamma}(r)=\tilde{F}_{d,\gamma}(1/r)$.  Using the change of variable $r\mapsto 1/r$ for the integral at \eqref{26241}, we get
\begin{equation*}\EXP[R^b]=\int_0^\infty C_{d,\gamma}\frac{r^{2a-2\gamma-b}}{(r^2+1)^a}\tilde{F}_{d,\gamma}(r)dr=\EXP[R^{\bar{b}}].\end{equation*}
(ii) For $d\ge 1$ and $b\in\R$, by Cauchy-Schwarz inequality, we have
\[\EXP[R^b]=\frac{1}{2}(\EXP[R^b]+\EXP[R^{\bar{b}}])=\int_0^\infty g_{d,\gamma}(r)\frac{r^b+r^{\bar{b}}}{2}dr\ge \int_0^\infty g_{d,\gamma}(r)r^{\frac{b+\bar{b}}{2}}dr=\int_0^\infty g_{d,\gamma}(r)r^{b^*}=\EXP[R^{b^*}].\]
For $d=1$, one obtains the following formula for the general moment $\EXP[R^b]$, with the help of Mathematica for example:
\begin{equation}\label{moment1d}
\EXP[R^b]=\frac{\Gamma (1-\gamma)^2 \Gamma (2 \gamma-1)^2 (\Gamma (b+2 \gamma-1) \Gamma (-b-2 \gamma+2)+\Gamma (b+4 \gamma-2) \Gamma (-b-4 \gamma+3))}{\sqrt{\pi } \Gamma \left(\frac{3}{2}-2 \gamma\right) \Gamma \left(\gamma-\frac{1}{2}\right)^2 \Gamma (-b-2 \gamma+2) \Gamma (b+4 \gamma-2)}.\end{equation}
For $b=b^*$, the above formula reduces to \eqref{moment}.

For the case $d\ge 2$,
by \eqref{210241} and \autoref{angle-distr-prop}, one can rewrite $\EXP[R^{b^*}]$ in terms of the angles as
\[\EXP[R^{b^*}]={\bar{C}_{d,\gamma}}\int_{0}^{\pi}\int_0^{\pi-\phi_1}(\sin\phi_2\,\sin(\phi_1+\phi_2))^{\frac{d-1}{2}-\gamma}\sin^{2\gamma-2}\phi_1d\phi_2d\phi_1.\]
(iii) In the case $(b,\gamma)\in(0,\infty)\times[\frac{d+3}{6},\frac{d+2}{4})\bigcup\,(-\infty,0)\times (\frac{1}{2},\frac{d+3}{6}]$, note that $b$ and $\bar{b}$ have different signs. There exists $p\in(1,\infty)$ such that $\frac{b}{p}+\frac{\bar{b}}{p'}=0$ where $\frac{1}{p}+\frac{1}{p'}=1$. The proof proceeds by noticing the logarithmic convexity of the moment function $x\mapsto\EXP[R^x]$ as follows. By Holder's inequality,
\[\mathbb{E}[{{R}^{b}}]=\mathbb{E}{{[{{R}^{b}}]}^{1/p}}\mathbb{E}{{[{{R}^{{\bar{b}}}}]}^{1/p'}}\ge \mathbb{E}[{{R}^{b/p+\bar{b}/p'}}]=\mathbb{E}[{{R}^{0}}]=1.\]
\end{proof}

\subsection{Explosion and hyperexplosion of $(d,\gamma)$-DSY cascades}
In this section, we establish criteria for non-explosion, explosion, non-hyperexplosion, and hyperexplosion for $(d,\gamma)$-DSY cascades. 

\begin{thm}\label{expl-crit}
Let $d\ge 1$ and $\gamma\in(\frac{1}{2},\frac{d+2}{4})$. The  $(d,\gamma)$-DSY cascade is 
\begin{enumerate}
\item[(a)] non-explosive if $\EXP[R^{b}]<1/2$ for some $b>0$,
\item[(b)] a.s.\ explosive if $\EXP[\ln R_{\max}]> 0$, where $R_{\max}=\max\{R_1,R_2\}$, $R_1=\frac{|W_1|}{|W_\root|}$, $R_2=\frac{|W_2|}{|W_\root|}$,
\item[(c)] non-hyperexplosive if $\EXP[\ln R_{\min}]<0$, where $R_{\min}=\min\{R_1, R_2\}$,
\item[(d)] a.s.\ hyperexplosive if $\EXP[R^{-2\gamma\alpha}]<\frac{1}{2}$ for some $\alpha\in(0,1]$.
\end{enumerate}
\end{thm}

\begin{proof}
Part (a) is \cite[Prop. 2.6]{part2_2021}. For Part (b), the proof is  similar to that of
\cite[Prop. 2.7]{part2_2021}. Namely, consider the "greedy" random path $\bar{s}\in \partial\T$
 which starts at the root $\root$ and, once at vertex $v$, follows either $v1$ or $v2$, whichever has a larger ratio of $R_{v1}$ and $R_{v2}$. 
The associated time is
\[\tau_{\bar{s}}=\sum_{j=0}^{\infty} \frac{T_{\bar{s} | j}}
{|W_{\bar{s}|j}|^{2\gamma}}
=\sum_{j=0}^{\infty} \frac{T_{\bar{s} | j}}{\prod_{k=0}^j R_{\max,k}^{2\gamma}}
\ge S,\] 
where $R_{\max,j+1}=\max\{R_{[\bar{s}|j]1},R_{[\bar{s}|j]2}\}$ are i.i.d.\ 
copies of $R_{\max}$.
Suppose $\EXP[\ln(R_{\max})]=\mu> 0$. By the Law of Large Numbers, 
\[\frac{\ln(\prod_{k=0}^j R_{\max,k})}{j}
=\frac{\sum_{k=0}^j \ln(R_{\max,k})}{j}\to \mu, \qquad \mbox{a.s. as}\ j\to\infty.\]
Therefore, a.s.\ there exists a random number $J\in \mathbb{N}$ such that for $j>J$
\[
\frac{\sum_{k=0}^j \ln(R_{\max,k})}{j}>\frac{\mu}{2}
\]
or equivalently,
\[
\frac{1}{\prod_{k=0}^j R_{\max,k}^{2\gamma}}<\frac{1}{\alpha^j},
\qquad \alpha=e^{\mu\gamma}>1.
\]
Since $\{T_{\bar{s} | j}\}_{j\in\mathbb{N}}$ are i.i.d.\ exponential 
random variables with intensity 1,
we have 
\begin{equation*}
\tau_{\bar{s}}\le \sum_{j=0}^{J} \frac{T_{\bar{s} | j}}{\prod_{k=0}^j 
R_{\max,k}^{2\gamma}}+ \sum_{j=J+1}^\infty\frac{T_{\bar{s} | j}}{\alpha^j}
\le  \sum_{j=0}^{J} \frac{T_{\bar{s} | j}}{\prod_{k=0}^j R_{\max,k}^{2\gamma}}+
\sum_{j=1}^{\infty} \frac{T_{\bar{s} | j}}{\alpha^j}<\infty 
\quad a.s.
\end{equation*}
since the first term is a finite sum and the second has a finite 
expected value a.s. Thus, $S<\infty$ a.s.

Part (c) is proven similarly to Part (b) except that this time one uses the greedy ray $\underline{s}$ that follows the sibling with a smaller ratio of $R_{v1}$ and $R_{v2}$ instead of the larger one. By the Law of Large Numbers, 
$\tau_{\underline{s}}$ is bounded from below by the a.s.\ divergent series $\sum_{j>J}{T_{\underline{s}|j}}\alpha^j$, where 
$\alpha=e^{\mu/2}$ with $\mu=-\EXP[\ln R_{\min}]>0$.

For Part (d), due to the self-similarity of the DSY cascade, one can assume 
$|\xi|=1$. Let 
\[L_n=\max\limits_{|v|=n}\sum\limits_{j=0}^n \frac{T_{v|j}}{|W_{v|j}|^{2\gamma}}\] 
be the maximum length of all rays up to generation $n$. One has $L=\lim\limits_{n\to\infty}L_n$ and 
\begin{equation}\label{Leq}L_{n+1}=T_\root+
\max\{L^{(1)}_nR_1^{-2\gamma},\,L^{(2)}_nR_2^{-2\gamma}\}.
\end{equation}
where $L_n^{(1)}$ and $L_n^{(2)}$ are independent copies of $L_n$. \eqref{Leq} implies
\begin{equation}\label{Leq2}
L_{n+1}\le T_\root+L^{(1)}_nR_1^{-2\gamma}+L^{(2)}_nR_2^{-2\gamma}.
\end{equation}
In view of the inequality
\[(a+b)^\alpha\le a^\alpha+b^\alpha\ \ \ \forall\,a,b>0,\,\alpha\in(0,1],\]
it follows from \eqref{Leq2} that for any 
$\alpha\in(0,1]$,
\begin{equation}\label{Leqa}
L^\alpha_{n+1}\le T^\alpha_\root+(L^{(1)}_n)^\alpha 
R_1^{-2\gamma\alpha}+(L^{(2)}_n)^\alpha R_2^{-2\gamma\alpha}.
\end{equation}
Take the expectation of both sides:
\begin{eqnarray*}
\EXP[L^\alpha_{n+1}]&\le &\EXP[T^\alpha_\root]+\EXP[(L^{(1)}_n)^\alpha 
R_1^{-2\gamma\alpha}]+\EXP[(L^{(2)}_n)^\alpha R_2^{-2\gamma\alpha}].
\end{eqnarray*}
By the self-similarity of the cascade, the vector $(L_n^{(1)},L_n^{(2)})$ is 
independent of the vector $(R_1,R_2)$. Therefore, the above inequality can be 
written as
\begin{eqnarray}
\label{eles}\EXP[L^\alpha_{n+1}]&\le& A + 2\EXP[R_1^{-2\gamma\alpha}]\, \EXP[L^\alpha_n].
\end{eqnarray}
By the estimate \eqref{eles}, the sequence $\{\EXP[L_n^{\alpha}]\}_{n\ge 0}$ 
is bounded. By Lebesgue's Monotone Convergence Theorem, $\EXP[L^\alpha]=\lim \EXP[L_n^\alpha]<\infty$. Therefore, $L<\infty$ almost surely.
\end{proof}


\begin{remark}\label{explicit-expectations}
The explosion, non-hypterexplosion, hyperexplosion criteria in Parts (b), (c), (d) of \autoref{expl-crit} can be straightforwardly adapted to apply to a general self-similar DSY cascade as defined in \cite{part2_2021}. Part (b) is an improvement of the explosion criterion in \cite[Thm 2.7]{part2_2021}.
\end{remark}




\begin{remark}
Let $\Phi_1$ and $\Phi_2$ be angles in the triangle with sides of lengths $1$, $R_1$, $R_2$, where $\Phi_1$ is opposite to $R_2$ and $\Phi_2$ is opposite to
$R_1$. By \autoref{angle-distr-prop} and the fact that $R_1=\frac{\sin\Phi_2}{\sin(\Phi_1+\Phi_2)}$, we have
\begin{align}
\label{ER-formula}\EXP[R_1^{-2\gamma\alpha}]=\bar{C}_{d,\gamma}\int\limits_0^\pi\int\limits_0^{\pi-\phi_1}(\sin\phi_1)^{2\gamma-2}(\sin\phi_2)^{2\gamma(1-\alpha)}(\sin(\phi_1+\phi_2))^{d+1-4\gamma+2\gamma\alpha}d\phi_2d\phi_1.
\end{align}
Since $\gamma\in\left(\frac{1}{2},\frac{d+2}{4}\right)$, the integral in \eqref{ER-formula} is finite if and only if $\alpha\in\left(0,1-\frac{1}{2\gamma}\right)$.
\end{remark} 

The next result, 
\autoref{exp-nonexp}, serves two purposes.  First, it identifies the optimal value of $\alpha$ in Part (d) of \autoref{expl-crit}. Second, it gives explicit parametric regions of $(d,\gamma)$ for which the criteria in \autoref{expl-crit} apply. Note that numerical implementations of $\EXP[\ln R_{\max}]$ and $\EXP[\ln R_{\min}]$ using  \eqref{ERmax-formula} and \eqref{ERmin-formula} indicates that the parametric regions of $(d,\gamma)$ corresponding to explosion and non-hyperexplosion are actually larger  (\autoref{dgdiagram})  than the ones identified analytically in Part (b) and (c) of \autoref{exp-nonexp} below.


\begin{thm}\label{exp-nonexp}
Let $d\ge 1$ and $\gamma\in(\frac{1}{2},\frac{d+2}{4})$. The $(d,\gamma)$-DSY cascade is
\begin{itemize}
\item[(a)] non-explosive if $d\ge 2$, $\gamma\in(\frac{1}{2},\frac{d+3}{6})$ 
and $\EXP[R^{b^*}]<1/2$, where $b^*=\frac{d+3}{2}-3\gamma$. Moreover, $b^*$ 
corresponds to the best choice of $b$ in \autoref{expl-crit} (a) in the sense 
that $\EXP[R^{b^*}]=\min_{b>0}\EXP[R^{b}]$.
\item[(b)] a.s.\ explosive if $d=1,\gamma\in(\frac{1}{2},\frac{3}{4})$ 
or $d\ge 2,\gamma\in[\frac{d+3}{6},\frac{d+2}{4})$,
\item[(c)] non-hyperexplosive if $d\ge 1,\gamma\in(\frac{1}{2},\frac{d+3}{6}]$,
\item[(d)] a.s.\ hyperexplosive if $d\ge 1$, 
$\gamma\in(\frac{d+3}{6},\frac{d+2}{4})$ and $\EXP[R^{b^*}]<1/2$. Moreover, $b^*$ 
corresponds to the best choice of $\alpha$ in \autoref{expl-crit} (c) in the sense 
that $\EXP[R^{b^*}]=\min_{\alpha\in(0,1]}\EXP[R^{-2\gamma\alpha}]$.
\end{itemize}
\end{thm}

\begin{remark}
The simple ranges of $\gamma$ in Part (b) and (c) are smaller than the ranges of 
$\gamma$ determined by the general criteria in \autoref{expl-crit} 
(b), (c) (see \autoref{numeric}).
\end{remark}
\begin{proof}[Proof of \autoref{exp-nonexp}]
Part (a) follows directly from \autoref{expl-crit} (a). 
The optimality of $b^*$ follows from \autoref{moments} (ii). For Part (d), note that 
$b^*=-2\gamma\alpha$ with $\alpha =\frac{3}{2}-\frac{d+3}{4\gamma}\in(0,1]$. By 
\autoref{expl-crit} (d), the DSY cascade is a.s.\ hyperexplosive. 
 
We now prove Part (b) and Part (c). First, consider the case $d=1$. By the explosion criterion \cite[Prop. 2.7]{part2_2021}, it sufficies to show that 
$\EXP[R_{\max}^{-2\gamma}]<1$.
Let $g_{R_{\max}}$ be the pdf of $R_{\max}$, which is given in \autoref{d=1}. Because 
$H(\cdot|1)$ is a probability density symmetric with respect to $1/2$, we have
\[\mathbb{E}[R_{\max }^{-2\gamma }]=\int\limits_{0}^{\infty }{{{r}^{-2\gamma }}g_{R_{\max}}(r)dr}=\frac{\int\limits_{0}^{\infty }{{{r}^{-2\gamma }}g_{R_{\max}}(r)dr}}{\int\limits_{-\infty }^{\infty }{H(r |1) }dr}=\frac{\int\limits_{1/2}^{\infty }{{{r}^{-2}}|1-r{{|}^{2\gamma -2}}dr}}{\int\limits_{1/2}^{\infty }{|r-{{r}^{2}}{{|}^{2\gamma -2}}dr}}.\]
We want to show that
\[\int\limits_{1/2}^{\infty }{{{r}^{-2}}|1-r{{|}^{2\gamma -2}}dr}<\int\limits_{1/2}^{\infty }{|r-{{r}^{2}}{{|}^{2\gamma -2}}dr}.\]
Using the changing of variable $r\mapsto 1/r$ on both sides, we can rewrite the inequality to be shown as
\begin{equation}\label{851}\int\limits_{0}^{2}{{{r}^{2-2\gamma }}|1-r{{|}^{2\gamma -2}}dr}<\int\limits_{0}^{2}{{{r}^{2-4\gamma }}|1-r{{|}^{2\gamma -2}}dr}.\end{equation}
The integrals on both sides are then split into $\int_0^1+\int_1^2$. Using the change of variable $r\mapsto 2-r$ for the integral $\int_1^2$, one obtains an inequality equivalent to \eqref{851} as follows.
\begin{equation}
\label{852}
\int\limits_{0}^{1}{[{{r}^{a}}+{{(2-r)}^{a}}]{{(1-r)}^{-a}}dr}<\int\limits_{0}^{1}{[{{r}^{b}}+{{(2-r)}^{b}}]{{(1-r)}^{-a}}dr}
\end{equation}
where $a=2-2\gamma\in(1/2,1)$ and $b=2-4\gamma\in(-1,0)$. To show \eqref{852}, it suffices to show that
\begin{equation}
\label{853}
r^a+(2-r)^a<r^b+(2-r)^b\ \ \ \forall r\in(0,1).
\end{equation}
Because $a\in(1/2,1)$, the function $f(x)=x^a$ is concave on $(0,\infty)$. Thus,
\[\text{LHS}\eqref{853}=f(r)+f(2-r)<2f\left(\frac{r+(2-r)}{2}\right)=2f(1)=2.\]
Because $b\in(-1,0)$, the function $g(x)=x^b$ is convex on $(0,\infty)$. Thus,
\[\text{RHS}\eqref{853}=g(r)+g(2-r)>2g\left(\frac{r+(2-r)}{2}\right)=2g(1)=2.\]
Therefore, \eqref{853} has been proven, and explosion in the case $d=1$ follows.

Next, consider the case $d\ge2$. Let $\Phi_1$ and $\Phi_2$ be angles in the triangle with sides of lengths $1$, $R_1$, $R_2$, where $\Phi_1$ is opposite to $R_2$ and $\Phi_2$ is opposite to
$R_1$. 
We note that $R_1>R_2$ when $\Phi_1<\Phi_2$. By the absolute continuity 
of the joint distribution of $(\Phi_1,\Phi_2)$ with respect to the Lebesgue measure 
on $\Delta=\{(\phi_1,\phi_2):\ \phi_1,\phi_2\ge0,\ \phi_1+\phi_2\le\pi\}$, we conclude that 
\begin{equation}\label{R-max-min}
\EXP[\ln R_{\max}]>\EXP[\ln R_1]>\EXP[\ln R_{\min}].
\end{equation}
Together with \autoref{symmetry1} (c), one has $\EXP[\ln R_{\max}]>\EXP[\ln R_1]\ge 0$ for $\gamma\in[\frac{d+3}{6},\frac{d+2}{4})$, and $\EXP[\ln R_{\min}]<\EXP[\ln R_1]\le 0$ for $\gamma\in(\frac{1}{2},\frac{d+3}{6}]$. Thus, Part (b) and (c) follow.
\end{proof}

\begin{remark}
For $\gamma< (d+3)/6$, which is when $\EXP[\ln R_1]<0$, the Law of Large Number shows that the Markov process $\{\bY_{s|n}\}_{n\ge 0}$ is non-explosive for any deterministic ray $s\in\partial\T$. That is,
\[\tau_{s}=\sum_{j=0}^{\infty}\bY_{s|j}=\sum_{j=0}^{\infty} \frac{T_{{s} | j}} {|W_{{s}|j}|^{2\gamma}}=\sum_{j=0}^{\infty} \frac{T_{{s} | j}}{\prod_{k=0}^j R_{s|k}^{2\gamma}}=\infty.\]
For $\gamma= (d+3)/6$ (which includes the standard Navier-Stokes equations), one also has $\tau_s=\infty$ because $\{\ln R_{s|n}\}_{n\ge 1}$ is a simple random walk on $\R$ (\autoref{symmetry1}).
\end{remark}

\begin{lem}\label{d=1}
For $d=1$ and $\gamma\in(1/2,3/4)$, the pdf of $R_{\max}=\max\{R_1,R_2\}$ and $R_{\min}=\min\{R_1,R_2\}$ are given by
\begin{eqnarray}
\label{d=1max}{{g}_{{{R}_{\max }}}}(r)&=&\left\{ \begin{array}{*{35}{r}}
   2{{c}_{1,\gamma }}|r-{{r}^{2}}{{|}^{2\gamma -2}} & \textup{if} & r>1/2  \\
   0 & \textup{if} & r\le 1/2,  \\
\end{array} \right.\\
\label{d=1min}g_{R_{\min}}(r)&=&\left\{ \begin{array}{*{35}{r}}
   2c_{1,\gamma}(r+r^2)^{2\gamma-2} & \textup{if} & r>1/2  \\
   2c_{1,\gamma}((r+r^2)^{2\gamma-2}+|r-r^2|^{2\gamma-2}) & \textup{if} & r\le 1/2.\\
\end{array} \right.
\end{eqnarray}
\end{lem}
\begin{proof}
Assume without loss of generality that $|\xi|=1$. Then $\xi=1$ or $\xi=-1$. In view of the even parity of $h$, it suffices to consider the case $\xi=1$. Then $R_1=|W_1|$ and $R_2=|1-W_1|$. Thus, $R_{\max}\ge 1/2$. For any $r\ge 1/2$, 
\begin{eqnarray*}
\mathbb{P}({{R}_{\max }}\le r)&=&\mathbb{P}({{R}_{1}}\le {{R}_{2}}\le r)+\mathbb{P}({{R}_{2}}\le {{R}_{1}}\le r)\\
&=&\mathbb{P}({{W}_{1}}\le 1/2,1-{{W}_{1}}\le r)+\mathbb{P}({{W}_{1}}\ge 1/2,{{W}_{1}}\le r)\\
&=&2\mathbb{P}(1/2\le {{W}_{1}}\le r)\\   
&=&2\int\limits_{1/2}^{r}{H(s|1)ds}.
\end{eqnarray*}
Therefore, the probability density function of $R_{\max}$ is given by
\[{{g}_{{{R}_{\max }}}}(r)=\left\{ \begin{array}{*{35}{r}}
   2H(r|1) & \text{if} & r>1/2  \\
   0 & \text{if} & r\le 1/2  \\
\end{array} \right.\]
which is equivalent to \eqref{d=1max}. Similarly,
\begin{eqnarray*}
\P(R_{\min}\ge r)=2\P(W_1\ge 1/2,|1-W_1|\ge r)=2(\P(1/2\le W_1\le 1-r)+\P(W_1\ge 1+r)).
\end{eqnarray*}
Taking the derivative of both sides, one gets
\[g_{R_{\min}}(r)=\left\{ \begin{array}{*{35}{r}}
   2H(1+r|1) & \text{if} & r>1/2  \\
   2(H(1-r|1)+H(1+r|1)) & \text{if} & r\le 1/2\\
\end{array} \right.\]
which is equivalent to \eqref{d=1min}.
\end{proof}
\subsection{Numerical results}\label{numeric}
Let $$\tilde{\triangle}=\{(\phi_1,\phi_2)\in\triangle:\ \phi_1\le\phi_2,\ \phi_1,\phi_2\ge0,\ \phi_1+\phi_2\le\pi\}.
$$
Let $\Phi_1$ and $\Phi_2$ be angles in the triangle with sides of lengths $1$, $R_1$, $R_2$, where $\Phi_1$ is opposite to $R_2$ and $\Phi_2$ is opposite to
$R_1$. 
We note that $R_1>R_2$ for $(\Phi_1,\Phi_2)\in\tilde{\triangle}$. By \autoref{angle-distr-prop} and the fact that $R_1=\frac{\sin\Phi_2}{\sin(\Phi_1+\Phi_2)}$, we have
\begin{eqnarray}\label{ERmax-formula}
\EXP[\ln R_{\max}]
&=&2\iint\limits_{\tilde{\triangle}} \ln\left(r_1\right)f_{d,\gamma}(\phi_1,\phi_2)\,d\phi_2d\phi_1\\
\nonumber&=&2\bar{C}_{d,\gamma}\int\limits_0^{\pi/2}\int\limits_{\phi_1}^{\pi-\phi_1}  \ln\left(\frac{\sin\phi_2}{\sin(\phi_1+\phi_2)}\right) (\sin\phi_1\sin\phi_2)^{2\gamma-2}{\sin }^{d+1-4\gamma}(\phi_1+\phi_2)\, d\phi_1 d\phi_2,
\end{eqnarray}
\begin{eqnarray}\label{ERmin-formula}
\EXP[\ln R_{\min}]
&=&2\iint\limits_{\tilde{\triangle}} \ln\left(r_2\right)f_{d,\gamma}(\phi_1,\phi_2)\,d\phi_2d\phi_1\\
\nonumber&=&2\bar{C}_{d,\gamma}\int\limits_0^{\pi/2}\int\limits_{\phi_1}^{\pi-\phi_1}  \ln\left(\frac{\sin\phi_1}{\sin(\phi_1+\phi_2)}\right) (\sin\phi_1\sin\phi_2)^{2\gamma-2}{\sin }^{d+1-4\gamma}(\phi_1+\phi_2)\, d\phi_1 d\phi_2.
\end{eqnarray}
\autoref{dgdiagram} is a $d\gamma$-diagram that  shows the regions of nonexplosion, explosion, and hyperexplosion described by the criteria in \autoref{expl-crit}. More specifically, we implement numerically the identities \eqref{ERmax-formula} for the explosion criterion, \eqref{ERmin-formula} for the non-hyperexplosion criterion, and \eqref{moment} for the non-explosion and hyperexplosion criterian. Although the dimension $d$ has to be an integer, the criteria involve estimates that remain valid for any real values $d>0$. We observe the following:
\begin{enumerate}[1)]
\item For each fixed $d$, the non-explosion region of $(d,\gamma)$ corresponds to low values of $\gamma$, while hyper-explosion corresponds to high values of $\gamma$. As $\gamma$ increases, there are transitions from non-explosion to explosion (but non-hyperexplosion), and to hyperexplosion.
\item For $d\in\{3,4,5\}$ and $\gamma=1$, the DSY cascade is explosive and non-hyperexplosive. This extends the result of \cite[Prop.\ 12.8]{part2_2021} where explosion is proven for $d=3$ and $\gamma=1$. For $d\ge 12$ and $\gamma=1$, the DSY cascade is non-explosive, consistent with the theory; \cite[Prop.\ 12.3]{part1_2021}.
\item For $d\in\{1,2\}$ and $\gamma\in(\frac{1}{2},\frac{d+2}{4})$, the DSY cascade is explosive. 
\item For $d=1$, the non-hyperexplosion criterion 
\[\EXP[\ln R_{\min}]=\int_0^\infty g_{R_{\min}}(r)\ln rdr<0\]
infers that the cascade is non-hyperexplosive for $\gamma\in(0.5,\,0.6849]$. Here, $g_{R_{\min}}$ is given by \eqref{d=1min}. 
The hyperexplosion criterion $\EXP[R^{b^*}]<0.5$ infers that the cascade is hyperexplosive for $\gamma\in [0.7117,\,0.75)$.
\end{enumerate}
\begin{figure}[h!]
\centering\includegraphics[scale=.65]{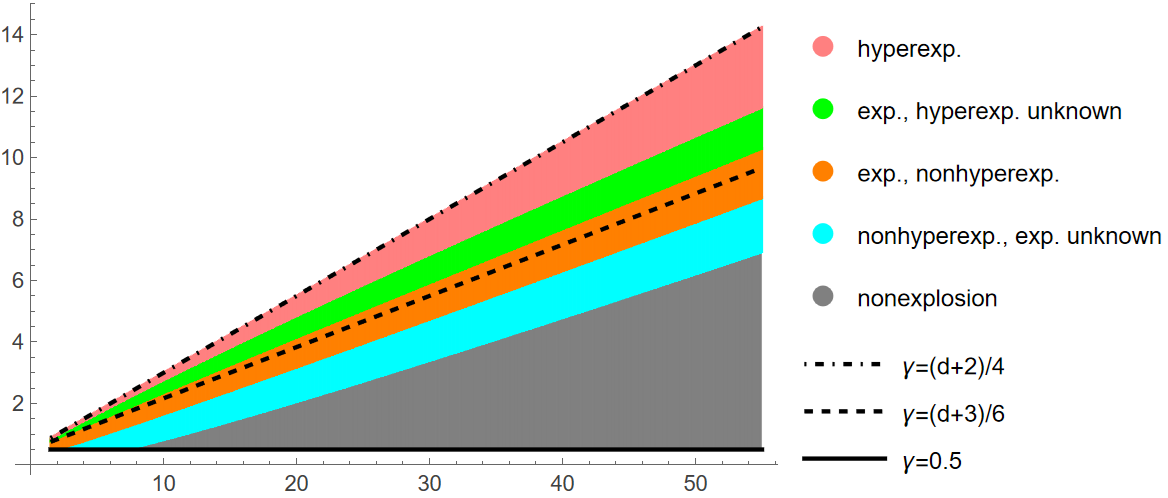}\\[3ex]
\includegraphics[scale=.65]{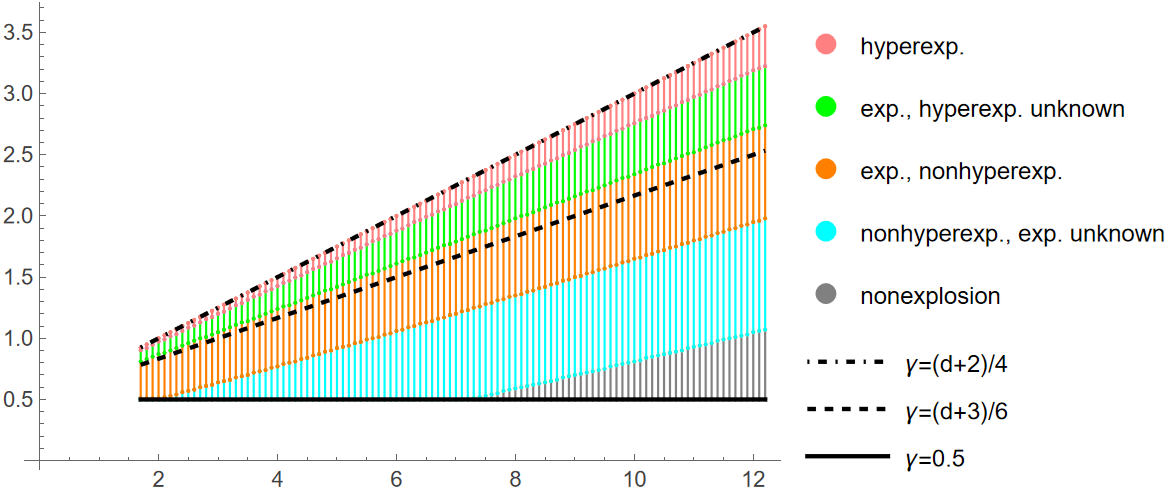}
\caption{$d\gamma$-diagram. The horizontal $d$-scale is discretized with stepsize 0.1, and the vertical $\gamma$-scale is discretized with stepsize 0.01. Two different ranges of $d$ are shown: $d\in[1.5,55]$ in the upper figure, and $d\in[1.7,12.2]$ in the lower figure.}
\label{dgdiagram} 
\end{figure}
\section{Majorizing Principle and the Fractional Montgomery-Smith equation}\label{reduction}
The quadratic form $a\odot_\xi b$ defined by \eqref{odotproduct} satisfies the following estimate:
\begin{equation}\label{odotestimate}
|a\odot_\xi b|\le\frac{1}{2}|a||b|\ \ \ \forall a,b\in\C^d,\,\xi\in\R^d\backslash\{0\}.
\end{equation}
The proof of \eqref{odotestimate} is elementary and can be found in \cite[Lem.\ 2.1]{dascaliuc2019jan}.

To establish the integrability of the solution process $\X$ defined by \eqref{eq:822191}, it is natural to use the estimate \eqref{odotestimate}. In replacing the vectorial product $a\odot_\xi b$ with the product of scalars, one is establishing a connection between \ref{nsedg} with the \emph{fractional Montgomery-Smith equation} in $\R^d$, $d\ge 1$:
\begin{equation}\label{MS}\tag*{$\textup{(MSE)}_{d,\gamma}$}
{{\partial }_{t}}v+(-\Delta)^\gamma v = (-\Delta )^{1/2}({{v}^{2}}),\ \ \ v(\xi,0)=v_0(\xi).
\end{equation}
Here, $v(\xi,t)$ and $v_0(\xi)$ are scalar valued functions.
Montgomery-Smith \cite{smith} introduced \ref{MS} for the case $d=3, \gamma=1$ as simple model to highlight the limitation of the perturbation method in establishing the well-posedness of the Navier-Stokes equations.  In our context, \ref{MS} plays two roles.  First, it serves as the governing equation for the probability of non-explosion by time $t$ of the $(d,\gamma)$-DSY cascade:
\begin{equation}\label{rhoeq}
\rho(\xi,t) = \P_\xi(S> t).
\end{equation}
Second, \ref{MS} serves as a tool to obtain a lower bound for the time $t_c$, defined by \eqref{tcritical} as the first time beyond which the solution process fails to be integrable.


\begin{prop}
\label{MSviaExplosion}
Let $d\ge 1$ and $\gamma\in(\frac{1}{2},\frac{d+2}{4})$. Then the function $\rho$ given by \eqref{rhoeq} satisfies
\begin{equation}
\label{explosionequation}
\rho(\xi,t) = e^{-|\xi|^{2\gamma} t} + \int_0^t |\xi|^{2\gamma} e^{-|\xi|^{2\gamma} s} \int_{\mathbb{R}^d} \rho(\eta,t-s)\rho(\xi-\eta,t-s) H(\eta|\xi) d\eta ds
\end{equation}
where $H$ is given by \eqref{Hkernel}. Consequently, $\rho$ is continuous in $t\ge 0$.
\end{prop}
\begin{proof}
The proof is similar to that of \autoref{819191}. By the inclusion $[Y_\root> t]\subset[S> t]$ and the the inheritance property of the non-explosion event, namely 
${\1_{S_\xi> t}}\1_{Y_\root\le t}=\1_{Y_\root\le t}(\1_{{{S}_{{{W}_{1}}}}> t-Y_\root}+\1_{{S_{{{W}_{2}}}}> t-Y_\root})$, the indicator $\mathcal{X}(\xi,t)=\1_{S_\xi>t}$ satisfies
\begin{equation}\label{pr_no_expl}
\mathcal{X}(\xi,t)=\left\{\begin{array}{ll}
1, & Y_\root> t,\\
\mathcal{X}^{(1)}(W_1, t-Y_\root)\mathcal{X}^{(2)}(W_2,t-Y_\root),\quad & Y_\root\le t,
\end{array}
\right.
\end{equation}
where $\mathcal{X}^{(1)}$ and $\mathcal{X}^{(2)}$ are two independent copies on $\mathcal{X}$.
Since $|\mathcal{X}|\le1$, we can take expectation in \eqref{pr_no_expl} to obtain that $\rho=\EXP\barX$ satisfies \eqref{explosionequation}.
The continuity of $\rho$ follows immediately from \eqref{explosionequation} since $\int_{\mathbb{R}^d} \rho(\eta,t-s)\rho(\xi-\eta,t-s) H(\eta|\xi) d\eta\le 1.$
\end{proof}
The constant function $1$ is a trivial solution of \eqref{explosionequation}.  According to \autoref{exp-nonexp}, there exists values of $(d,\gamma)$ for which the $(d,\gamma)$-DSY cascade is explosive. In such a case, $\rho(\xi,t)<1$ for some $(\xi,t)$ and \eqref{explosionequation} has a nonconstant solution. This observation is formalized as a nonuniqueness result of \ref{MS} in \autoref{nonuniquenessMS}.  

Note that \ref{MS} has the same scaling property as \ref{nsedg}.
The Fourier transform of \ref{MS} is given 
by
$$
\frac{\partial \hat{v}}{\partial t} +  |\xi|^{2\gamma} \hat{v}= c_0 |\xi| \int_{\mathbb{R}^d} \hat{v}(\eta, t) \hat{v}(\xi-\eta,t) d\eta,~~~~\hat{v}(\xi,0) = \hat{v}_0(\xi)
$$
where $c_0=(2\pi)^{-d/2}$. The mild formulation of this equation is
\begin{equation}\label{FMS}
\hat{v}(\xi,t)=e^{-t|\xi|^{2\gamma}}\hat{v}_0(\xi) + c_0\int_{0}^{t}|\xi |{{{e}^{-s|\xi {{|}^{2\gamma}}}}\int_{{{\mathbb{R}}^{d}}}{\hat{v}(\eta ,t-s)\hat{v}(\xi -\eta ,t-s)d\eta ds}}.
\end{equation}

\begin{cor}
\label{nonuniquenessMS}
Let $(d,\gamma)$ be such that the $(d,\gamma)$-DSY cascade is explosive and let $\rho(\xi,t)=\P_\xi(S>t)$.   
Then 
$$\hat{v}_1 (\xi,t) = c_0^{-1} h(\xi)~~~\textrm{and}~~~ \hat{v}_2(\xi,t) = {c_0^{-1} h(\xi)}\rho(\xi,t)$$ are two solutions of \eqref{FMS} with the same initial data $\hat{v}_0(\xi)=c^{-1}_0 h(\xi).$
\end{cor}

\begin{proof}
We give a proof for the case of $\hat{v}_2$ as the case of $\hat{v}_1$ follows similarly.  Rewriting \eqref{explosionequation} in terms of $\hat{v}_2$, we get
$$
\hat{v}_2 (\xi,t)\frac{c_0} { h(\xi )}= 
e^{-t|\xi|^{2\gamma}} + c_0^2
\int_0^t |\xi|^{2\gamma} e^{-|\xi|^{2\gamma} s} \int_{\mathbb{R}^d}\frac{ \hat{v}_2(\eta,t-s)}{ h(\eta)}  \frac{\hat{v}_2(\xi-\eta,t-s)}{h(\xi-\eta)} H(\eta|\xi) d\eta ds.
$$
Recalling the definition of $H(\eta|\xi)$ given in \eqref{Hkernel}, we obtain
$$
\hat{v}_2 (\xi,t)\frac{c_0} { h(\xi )}= 
e^{-t|\xi|^{2\gamma}} + \frac{c_0^2}{h(\xi)}
\int_0^t |\xi| e^{-|\xi|^{2\gamma} s} \int_{\mathbb{R}^d}{ \hat{v}_2(\eta,t-s)}  {\hat{v}_2(\xi-\eta,t-s)} d\eta ds
$$
which is equivalent to \eqref{FMS} with initial data $h(\xi)/c_0$. 
\end{proof}
It is clear that if $\hat{v}$ satisfies \eqref{FMS} then $\psi=c_0\hat{v}/h$ satisfies a \emph{normalized Fourier-transformed Montgomery-Smith equation}:
\begin{equation}\label{nFMS}
\psi (\xi ,t)={{e}^{-t|\xi {{|}^{2\gamma}}}}{{\psi }_{0}}(\xi)+\int_{0}^{t}{{{e}^{-s|\xi {{|}^{2\gamma}}}}|\xi {{|}^{2\gamma}}\int_{{{\mathbb{R}}^{d}}}{\psi (\eta ,t-s)\psi (\xi -\eta ,t-s)H(\eta |\xi )d\eta ds }}.\end{equation}
\autoref{nonuniquenessMS} affirms that $\psi_1(\xi,t)\equiv 1$ and $\psi_2(\xi,t)=\rho(\xi,t)$ are two solutions to \eqref{nFMS} with initial condition $\psi_0(\xi)\equiv 1$. They are distinct if and only if the $(d,\gamma)$-DSY cascade is explosive.

Equation \eqref{nFMS} can be viewed as a scalar analog of \eqref{mildnse} and is a mean flow of a solution process $\bar{\X}$ defined by
\begin{equation}\label{eq:819191}\barX(\xi,t)=\left\{ \begin{array}{*{35}{l}}
   \psi_0(\xi) & \text{if} & Y_\root> t,  \\
   \barX^{(1)}(W_1, t-Y_\root)\barX^{(2)}(W_2, t-Y_\root) & \text{if} & Y_\root\le t.  \\
\end{array} \right.\end{equation}
By replacing $\barX$ with $\barX\1_{[S>t]}$, we can assume that the recursion \eqref{eq:819191} eventually stops and thus, $\barX$ is well-defined. One can notice the similarity between \eqref{eq:819191} and \eqref{eq:822191} and the fact that \eqref{nFMS} and \eqref{mildnse} have the same underlying $(d,\gamma)$-DSY cascade. The inequality $|a\odot_\xi b|\le \frac{1}{2}|a||b|$ makes it possible to compare the solution to \eqref{mildnse} with initial condition $\chi_0$ with the solution to \eqref{nFMS} with initial condition $\psi_0=|\chi_0|/2$. This is known as a \emph{majorizing principle} (see \cite{lejan, dascaliuc2019jan} for the case $\gamma=1$).

\begin{prop}\label{69211}
Let $\psi_0:\mathbb{R}^d\to[0,\infty]$ be a measurable function. We have the following statements.
\begin{enumerate}[(a)]
\item The function $\psi(\xi,t)=\mathbb{E}_\xi\barX(\xi,t)\in [0,\infty]$ satisfies {\eqref{nFMS}} with initial condition $\psi(\cdot,0)=\psi_0$ \textup{(}with the convention $0\cdot\infty=0$\textup{)}.
\item For each $\xi\neq 0$, $e^{t|\xi|^2}\psi(\xi,t)$ is nondecreasing in $t$.
\item \textup{(Majorizing principle)}
Let $\chi_0:\mathbb{R}^d\to\mathbb{C}^d$ be a measurable function such that $|\chi_0|\le 2\psi_0$ and $\X$ be given by \eqref{eq:822191}. Suppose that for some $T\in(0,\infty]$, $\psi<\infty$ a.e.\ on $\mathbb{R}^d\times(0,T)$. Then there exists a subset $D\subset \mathbb{R}^d$ with full measure such that both $\psi(\xi,t)=\mathbb{E}_\xi\barX$ and $\mathbb{E}_\xi[|{\mathbf{X}}|\mathbbm{1}_{S>t}]$ are finite on $D_T=D\times[0,T)$. 
Moreover, $|\chi|\le 2\psi$ on $D_T$, where $\chi(\xi,t)=\EXP_\xi[\bX\1_{S>t}]$.
\end{enumerate}
\end{prop}  
\begin{proof}
The proof of Part (a) follows the same lines as that of \autoref{819191} and will be skipped. To show Part (b), we multiply both sides of \eqref{nFMS}$_{\psi_0}$ by $e^{t|\xi|^{2\gamma}}$ and get 
\begin{equation}\label{529211}e^{t|\xi|^{2\gamma}}\psi-\psi_0(\xi)=\int_0^t|\xi|^{2\gamma}e^{s|\xi|^{2\gamma}}f(\xi,s)ds\end{equation} where $f(\xi,s)=\int_{\R^d}\psi(\eta,s)\psi(\xi-\eta,s)H(\eta|\xi)d\eta\ge 0$. The monotonicity of $e^{t|\xi|^{2\gamma}}\psi$ immediately follows. 

For Part (c), we first show the existence of a subset $D\subset\mathbb{R}^d$ with full measure such that $\psi$ is finite on $D_T$. Let $A=\{(\xi,t)\in\mathbb{R}^d\times(0,T): \psi(\xi,t)=\infty\}$. For each $t\in[0,T)$, let $A_t=\{\xi\in\mathbb{R}^d: \psi(\xi,t)=\infty\}$. By Fubini's theorem, $0=m(A)=\int_0^T m(A_t)dt$, where $m$ denotes the Lebesgue measure. This implies $A_t$ is null set for a.e.\ $t\in[0,T)$. By Part (b), $A_t\subset A_s$ whenever $t<s$. 
Consequently, $m(A_t) = 0$ for all $t\in[0,T)$. The monotonicity of $(A_t)_{t\in[0,T)}$ implies that the union $\cup_{t\in(0,T)}A_t$ is actually a countable union of null sets, which is also a null set. One can choose $D=\mathbb{R}^d\backslash\cup_{t\in[0,T)}A_t$. 
Next, it can be observed from \eqref{eq:822191} that on the nonexplosion event $[S> t]$, $\bX$ is an $\odot$-product of $\chi_0(W_v)$, $v\in\partial V(\xi,t)$, in a suitable order, where $\partial V(\xi,t)$ is given by \eqref{tleaves}.  By the inequality $|a\odot_\xi b|\le\frac{1}{2}|a||b|$,  one has 
\begin{align*}|\X(\xi,t)|&\le\frac{1}{2^{|\partial V(\xi,t)|-1}}\prod_{v\in\partial V(\xi,t)}|\chi_0(W_v)|\1_{S>t}\\
&\le\frac{1}{2^{|\partial V(\xi,t)|-1}}\prod_{v\in\partial V(\xi,t)}2\psi_0(W_v)\1_{S>t}=2\prod_{v\in\partial V(\xi,t)}\psi_0(W_v)\1_{S>t}=2\bar{\X}(\xi,t).
\end{align*}
Therefore,
\[\EXP_\xi[|\bX(\xi,t)|\1_{S>t}]\le 2\mathbb{E}_\xi[\barX(\xi,t)]<\infty\ \ \ \forall (\xi,t)\in D\times[0,T).\]
\end{proof}
\begin{remark}
\label{tclowerbound}
According to \autoref{69211}(c), if \eqref{FMS} has a mild solution on $(0,\tau)$ then $\mathbb{E}_\xi[|\bX|\mathbbm{1}_{S>t}]<\infty$ a.e.\ on $\R^d\times(0,\tau)$ which implies the lower bound on the critical time $t_c\ge \tau$. Part (i) of the following proposition (due to Ferreira and Villamizar-Roa \cite{ferreira}) asserts the local-in-time existence and uniqueness of mild solutions to \eqref{FMS} for initial data $v_0$ in the pseudomeasure-type space $PM^a$.
\end{remark}
For any $a>0$, let $PM^a$ be the space of pseudo-measures defined by \eqref{pma} and whose norm is 
\begin{equation}
\|f\|_{PM^a}=\||\xi|^a\hat{f}\|_{L^\infty}.
\end{equation}
Denote by $BC((0,T);X)$ the space
of bounded and continuous continuous functions from $(0,T)$ to a Banach space $X$.
\begin{prop}\label{mspmsol}
For $d\ge 1$, $\gamma\in(\frac{1}{2},\frac{d+2}{4})$, let $a_*=d+1-2\gamma>0$ and $\kappa=\kappa_{d,\gamma}=\frac{c_{d,\gamma}}{c_0}$, where $c_0=(2\pi)^{-d/2}$ and $c_{d,\gamma}$ is given by \eqref{constantc}. We have the following statements. 
\begin{enumerate}[(i)]
\item If $v_0\in PM^a$ for $a\in(a_*,d)$ then there exists a number $\tau\in(0,\infty)$ such that \eqref{FMS} has a unique solution $v\in BC((0,\tau),PM^a)$.
\item If $v_0\in PM^{a_*}$ and $\|v_0\|_{PM^{a_*}}\le \kappa$, then \eqref{FMS} has a global solution $v\in L^\infty((0,\infty),PM^{a_*})$ with $\|v\|_{L^\infty_t PM^{a_*}}\le \kappa$.
\item If $v_0\in PM^{a_*}$ and $\|v_0\|_{PM^{a_*}}< \kappa$, \eqref{FMS} has a unique global solution $v$ in the ball $\{w:\|w\|<\kappa\}$ of $L^\infty((0,\infty),PM^{a_*})$. Moreover, 
\begin{equation}\label{decay}|\hat{v}(\xi,t)|\le ce^{-\delta|\xi|t^{1/(2\gamma)}}|\xi|^{-a_*}\ \ \ \forall(\xi,t)\end{equation}
where the constants $c,\delta>0$ depend on $\gamma$ and $\|v_0\|_{PM^{a_*}}$.
\item For any $\epsilon>0$, let $v_{0,\epsilon}$ be the function whose Fourier transform is $\hat{v}_{0,\epsilon}(\xi)=(1+\epsilon)\kappa |\xi|^{-a_*}$, which implies $\|v_{0,\epsilon}\|_{PM^{a_*}}=(1+\epsilon)\kappa$. If the $(d,\gamma)$-DSY cascade is non-explosive then \eqref{FMS} has no solution in $L^\infty((0,\tau),PM^{a_*})$ for any $\tau>0$.
\end{enumerate}
\end{prop}
\begin{remark}
In Part (iv), one can express $v_{0,\epsilon}$ in the physical space using \eqref{fourier} as 
\begin{equation}\label{ctilde}
v_{0,\epsilon}(x)=(1+\epsilon)\tilde{c}_{d,\gamma}|x|^{1-2\gamma},\ \ \ \ \ \ \tilde{c}_{d,\gamma}=2^{2\gamma-1}\frac{\Gamma(2\gamma-1)\Gamma(\frac{d+1}{2}-\gamma)}{\Gamma(\frac{d+2-4\gamma}{2})\Gamma(\gamma-\frac{1}{2})}.
\end{equation}
\end{remark}
\begin{proof}[Proof of \autoref{mspmsol}]
For Part (i), the local existence and uniqueness of solutions is guaranteed by applying \cite[Theorem 1.2, 1.3]{ferreira} for the special case $a=b$ and $\theta=1$. Strictly speaking, \cite{ferreira} deals with the fractional Navier-Stokes equations rather than the fractional Montgomery-Smith equation. However, their proof does not use any special structure of the $\odot$-product other than the fact that $|w_1\odot_\xi w_1|\le |w_1||w_2|$ for all $w_1,w_2\in\R^d$. Thus, the results also hold for the fractional Montgomery-Smith equation.

For Part (ii), observe that 
\[\|\psi_0\|_{L^\infty}=\frac{c_0}{c_{d,\gamma}}\|v_0\|_{PM^{a_*}}=\frac{1}{\kappa}\|v_0\|_{PM^{a_*}}\le 1\]
where $\psi_0=c_0\hat{v}_0/h$. Let $\barX$  be the minimal solution process for \eqref{nFMS}. On the event $[S>t]$, one has 
\begin{equation}\label{Xclosed}
\barX(\xi ,t) = \prod\nolimits_{v \in \partial V(\xi ,t)} {{\psi _0}({W_v})}\end{equation} 
and therefore, $|\barX(\xi,t)|\le 1$ for all $\xi,t$. The function $\psi(\xi,t)=\EXP_\xi\barX(\xi,t)$ is a global solution to \eqref{nFMS}. Thus, the inverse Fourier transform $v=\frac{1}{c_0}\mathcal{F}^{-1}\{{\psi h}\}$ is a global solution to \eqref{FMS} and satisfies $\|v\|_{L^\infty_t PM^{a_*}}=\kappa\|\psi\|_{L^\infty}\le \kappa$.

For Part (iii), let $v_0\in PM^{a_*}$ with $\|v_0\|_{PM^{a_*}}<\kappa$, and let $v$ be the global solution defined in Part (ii). Then $\|\psi_0\|_{L^\infty}=\frac{1}{\kappa}\|v_0\|_{PM^{a_*}}<1$ and 
\[|\barX(\xi ,t)|= \prod\limits_{v \in \partial V(\xi ,t)} |{{\psi _0}({W_v})}|\1_{S>t}\le \|\psi_0\|_{L^\infty}^{|\partial V(\xi,t)|}\le \|\psi_0\|_{L^\infty}<1.\]
Thus, $$\|v\|_{L^\infty_t PM^{a_*}}=\kappa\|\psi\|_{L^\infty}=\kappa\|\EXP_\xi\barX(\xi,t)\|_{L^\infty}<\kappa.$$
Let $\tilde{v}$ be another solution to \eqref{FMS} such that $\|\tilde{v}\|_{L^\infty_t PM^{a_*}}<\kappa$.  To show that $\tilde{v}=v$, we use "ground state" argument as in \cite{lejan, alphariccati,dascaliuc2019jan}. Let $\tilde{\psi}=c_0\mathcal{F}\{\tilde{v}\}/h$, which is a solution to \eqref{nFMS} satisfying $\|\tilde{\psi}\|_{L^\infty}<1$.
Consider a sequence of stochastic processes $\tilde{\bX}_n$ defined inductively by $\tilde{\bX}_0(\xi,t)=\tilde{\psi}(\xi,t)$ and for $n\ge 0$,
\begin{equation}\label{tilX}\tilde{\bX}_{n+1}(\xi,t)=\left\{ \begin{array}{*{35}{l}}
   \psi_0(\xi) & \text{if} & Y_\root> t,  \\
   \tilde{\bX}_n^{(1)}(W_1, t-Y_\root)\tilde{\bX}_n^{(2)}(W_2, t-Y_\root) & \text{if} & Y_\root\le t.  \\
\end{array} \right.\end{equation}
By conditioning on $T_\root$, one has
\[\psi^{(n+1)} (\xi ,t)={{e}^{-t|\xi {{|}^{2\gamma}}}}{{\psi }_{0}}(\xi)+\int_{0}^{t}{{{e}^{-s|\xi {{|}^{2\gamma}}}}|\xi {{|}^{2\gamma}}\int_{{{\mathbb{R}}^{d}}}{\psi^{(n)} (\eta ,t-s)\psi^{(n)} (\xi -\eta ,t-s)H(\eta |\xi )d\eta ds }}\]
where ${\psi}^{(n)}(\xi,t)=\EXP_{\xi}\tilde{\bX}_n(\xi,t)$. Since $\psi^{(0)}=\tilde{\psi}$ is a solution to \eqref{nFMS}, one can prove by induction that $\psi^{(n)}=\tilde{\psi}$ for all $n$.  On the non-explosion event $[S>t]$,  $\tilde{\bX}_n=\mathbf{X}$ for sufficiently large random $n$, while on the explosion event $[S\le t]$, 
\begin{equation}\label{Xnexplicit}
\tilde{\bX}_n(\xi,t)\1_{S\le t}=\left(\prod_{v\in\partial V(\xi,t)}\psi_0(W_v)\right)\prod_{v\not\in\partial V(\xi,t),\,|v|=n}\tilde{\psi}\left(W_v,t-\sum_{j=0}^{|v|-1}Y_{v|j}\right).
\end{equation}
Let $\alpha=\max\{\|\psi_0\|_{L^\infty},\|\tilde{\psi}\|_{L^\infty}\}<1$. Then \eqref{Xnexplicit} implies $|\tilde{\bX}_n(\xi,t)|\1_{[S<t]}\le\alpha^{2^n}$, which converges to 0. Therefore, $\lim \tilde{\bX}_n(\xi,t)=\bX(\xi,t)\1_{S>t}$.
By Lebesgue's Dominated Convergence Theorem,
\[
\tilde{\psi}=\lim\limits_{n\to\infty}\mathbb{E}\tilde{\bX}_n=\mathbb{E}_\xi[{\mathbf{X}\1_{[S>t]}}]=\psi.
\]
Next, we show \eqref{decay}. One can infer from the representation \eqref{Xclosed} that $|\psi(\xi,t)|\le\sum_{n=1}^\infty\alpha^np_n(\xi,t)$ where
\begin{equation}\label{probability}
p_n(\xi,t)=\mathbb{P}_{\xi}(S>t, |\partial V(\xi,t)|=n).\end{equation}
By conditioning on the first time of branching, one gets
\begin{equation}\label{probrecursion}
{{p}_{n}}(\xi ,t)=\int_{0}^{t}{|\xi {{|}^{2\gamma}}{{e}^{-s|\xi {{|}^{2\gamma}}}}\int_{{{\mathbb{R}}^{d}}}{\sum\limits_{k=1}^{n-1}{{{p}_{k}}(\eta ,t-s){{p}_{n-k}}(\xi -\eta ,t-s)H(\eta |\xi )d\eta ds}}}.\end{equation}
Let $\beta=\frac{1}{2\gamma}<1$ and $b=b(\gamma)=(2\beta)^{\beta/(1-\beta)}(1-2\beta)$. 
We claim that
\begin{equation}\label{lem_beta}
x^\beta\le\frac{x}{2}+b\ \ \ \forall x\ge 0.
\end{equation}
Let $f(x)=x/2-x^\beta$. Then $f'(x)=1/2-\beta x^{\beta-1}$. The minimum of $f$ is $-b$ and is attained at $x=(2\beta)^{1/(1-\beta)}$. Thus, the claim is proved. 

Note that $p_1(\xi,t)=e^{-|\xi|^{2\gamma} t}$. Applying \eqref{lem_beta} for $x=|\xi|^{2\gamma}t$, we get $p_1(\xi,t)\le e^{b-|\xi|t^\beta}$. We claim that 
\begin{equation}\label{exp_ineq}
{{e}^{-|\xi {{|}^{2\gamma}}s}}{{e}^{-|\eta |(t-s)^\beta}}{{e}^{-|\xi -\eta |(t-s)^\beta}}\le {{e}^b}{{e}^{-|\xi {{|}^{2\gamma}}s/2}}{{e}^{-|\xi |t^\beta}}\ \ \ \forall\xi,\eta\in\R^d,\,0<s<t.\end{equation}
This inequality is equivalent to
\begin{equation}\label{vect_ineq}
|\xi|t^\beta-(|\eta|+|\xi-\eta|)(t-s)^\beta\le b+\frac{|\xi|^{2\gamma}s}{2}.
\end{equation}
By the triangle inequality and the fact that $\beta\in(0,1)$, we have LHS\eqref{vect_ineq}$\le |\xi|(t^\beta-(t-s)^\beta)\le|\xi|s^\beta$. Applying \eqref{lem_beta} for $x=|\xi|^{2\gamma}s$, we get $|\xi|s^\beta\le$\,RHS\eqref{vect_ineq}. Thus, \eqref{exp_ineq} is proved.

Using the estimate \eqref{exp_ineq}, one can probe by induction on $n\ge 1$ that 
\[p_n(\xi,t)\le\theta\lambda^{n-1}C_ne^{-|\xi|t^\beta}\ \ \ \forall\,n\in\mathbb{N},\]
where $\theta=e^{b}$, $\lambda=2e^{2b}$, and $\{C_n\}$ is the Catalan sequence $C_1=1$, $C_n=\sum_{k=1}^{n-1}C_kC_{n-k}$. Note that $C_n$ grows asymptotically as $4^n$. Since $p_n\in[0,1]$, for any $\delta\in(0,1)$,
\[{{p}_{n}}(\xi ,t)\le {{\left( \theta {{\lambda }^{n-1}}{{C}_{n}}{{e}^{-|\xi |t^\beta}} \right)}^{\delta }}\lesssim {{(4\lambda )}^{\delta n}}{{e}^{-\delta |\xi |t^\beta}}.\]
Choose $\delta>0$ small enough such that $\mu\alpha<1$, where $\mu=4^\delta\lambda^\delta$. Then
\[\psi (\xi ,t)\lesssim \sum\limits_{n=1}^{\infty }{{{({{4}^{\delta }}{{\lambda }^{\delta }}\alpha )}^{n}}{{e}^{-\delta |\xi |t^\beta}}}= \sum\limits_{n=1}^{\infty }{{{(\mu \alpha )}^{n}}{{e}^{-\delta |\xi |t^\beta}}}=\frac{\mu \alpha }{1-\mu \alpha }{{e}^{-\delta |\xi |t^\beta}}.\]
This inequality infers \eqref{decay}.

For Part (iv), let us fix $\epsilon>0$ and drop the subscript $\epsilon$ from  $v_{0,\epsilon}$. Suppose by contradiction that \ref{MS} with initial data $v_0$ has a solution $v\in L^\infty((0,\tau),PM^{a_*})$ for some $\tau>0$. Let $\phi_0=c_0\hat{v}_0/h\equiv 1+\epsilon$ and $\phi=c_0\hat{v}/h$. Then 
\begin{equation}\label{eqphi}
\|\phi\|_{L^\infty(\R^d\times(0,\tau))}<\infty\end{equation}
and $\phi=F[\phi]$, where
\begin{eqnarray*}
F[\phi](\xi,t)={{e}^{-t|\xi {{|}^{2\gamma}}}}{{\phi }_{0}}(\xi)+\int_{0}^{t}{{{e}^{-s|\xi {{|}^{2\gamma}}}}|\xi {{|}^{2\gamma}}\int_{{{\mathbb{R}}^{d}}}{\phi (\eta ,t-s)\phi (\xi -\eta ,t-s)H(\eta |\xi )d\eta ds }}.
\end{eqnarray*}
Let $\phi^{(0)}\equiv 0$ and $\phi^{(n+1)}=F[\phi^{(n)}]$ for all $n\ge 0$. By induction on $n$, one can show that $\phi^{(n)}\le \phi$ for all $n\ge 0$. Let $\bZ$ be the minimal solution process given by
\begin{equation}\label{Zproc}\bZ(\xi,t)=\left\{ \begin{array}{*{35}{l}}
   1+\epsilon & \text{if} & Y_\root> t,  \\
   \bZ^{(1)}(W_1, t-Y_\root)\bZ^{(2)}(W_2, t-Y_\root) & \text{if} & Y_\root\le t.  \\
\end{array} \right.\end{equation}
Define a stochastic recursion $\bZ_n$ as follows: $\bZ_0\equiv 0$ and for $n\ge 0$,
\begin{equation}\label{Znproc}
\bZ_{n+1}(\xi,t)=\left\{ \begin{array}{*{35}{l}}
   1+\epsilon & \text{if} & Y_\root> t,  \\
   \bZ_n^{(1)}(W_1, t-Y_\root)\bZ_n^{(2)}(W_2, t-Y_\root) & \text{if} & Y_\root\le t.  \\
\end{array} \right.\end{equation}
Since $\bZ$ is a minimal solution process, $\lim \bZ_n(\xi,t) = \bZ(\xi,t)$ for all $(\xi,t)$. By conditioning on $Y_\root$, one has
\[\EXP[\bZ_{n+1}(\xi,t)]={{e}^{-t|\xi {{|}^{2\gamma}}}}(1+\epsilon)+\int_{0}^{t}{{{e}^{-s|\xi {{|}^{2\gamma}}}}|\xi {{|}^{2\gamma}}\int_{{{\mathbb{R}}^{d}}}{\EXP[\bZ_n (\eta ,t-s)]\EXP[\bZ_n (\xi -\eta ,t-s)]H(\eta |\xi )d\eta ds }}.\]
Clearly, $\EXP[\bZ_0(\xi,t)]=0=\phi^{(0)}$. By induction on $n$, one get $\EXP\bZ_n=\phi^{(n)}$ for all $n\ge 0$. By Fatou's lemma,
\begin{equation}\label{phiineq}\phi\ge\liminf \phi^{(n)}=\liminf\EXP\bZ_n\ge \EXP\bZ.\end{equation}
Let $\tilde{\bZ}$ be the minimal solution process given by
\begin{equation}\label{Ztildeproc}\tilde{\bZ}(\xi,t)=\left\{ \begin{array}{*{35}{l}}
   (1+\epsilon)^{-1} & \text{if} & Y_\root> t,  \\
   \tilde{\bZ}^{(1)}(W_1, t-Y_\root)\tilde{\bZ}^{(2)}(W_2, t-Y_\root) & \text{if} & Y_\root\le t.  \\
\end{array} \right.\end{equation}
Let $w_0$ be a function such that $\hat{w}_0=(1+\epsilon)^{-1}h/c_0$. Then $\|w_0\|_{PM^{a_*}}=(1+\epsilon)^{-1}<1$. According to Part (iii), the function $w=w(\xi,t)$ whose Fourier transform is $\hat{w}=\EXP[\tilde{\bZ}(\xi,t)]h(\xi)/c_0$ satisfies
\[\hat{w}(\xi,t)\le ce^{-\delta|\xi|t^{1/(2\gamma)}}|\xi|^{-a_*}\ \ \ \forall(\xi,t)\]
for some constants $c,\delta>0$. Consequently, $\EXP[\tilde{\bZ}(\xi,t)]\lesssim e^{-\delta|\xi|t^{1/(2\gamma)}}$. Because the DSY cascade is non-explosive, $\bZ(\xi,t)\tilde{\bZ}(\xi,t)=1$. By \eqref{phiineq} and Schwarz's inequality,
\[\phi(\xi,t)\ge\EXP[\bZ(\xi,t)]\ge \frac{1}{\EXP[\tilde{\bZ}(\xi,t)]}\gtrsim e^{\delta|\xi|t^{1/(2\gamma)}}.\]
This contradicts the fact that $\|\phi\|_{L^\infty(\R^d\times(0,\tau))}<\infty$.
\end{proof}
\begin{remark}
Le Jan and Sznitman thinned the stochastic cascade to avoid the stochastic explosion issue. This procedure simplifies the proof of uniqueness of solutions \cite[Thm 2.4]{lejan}. The same technique was used in \cite[Thm 3.2]{rabi}. However, the side effect of thinning is that it reduces the size of the ball in which the uniqueness can be proven. Specifically, they prove the uniquess in the ball of radius $\kappa/2$ only. In contrast, one can see in the proof of Part (iii) that the stochastic explosion gives a natural mechanism to obtain the convergence $\barX_n\to 0$ on the explosion event $[S\le t]$, which leads to the uniqueness of solution in the ball of radius $\kappa$.
\end{remark}
\begin{remark}
Part (iv) shows that there are no minimal blowup initial data for \eqref{FMS} in the critical space $PM^{a_*}$. A global solution exists as long as $v_0$ is in the closed ball of radius $\kappa$. But in any neighborhood of the ball in $PM^{a_*}$, one can find an initial data $v_0$ such that the corresponding solution blows up instantly. Note that our method requires the non-explosion of the $(d,\gamma)$-DSY cascade to obtain such a sharp result. In any $(d,\gamma)$-DSY cascade, one can construct a sufficiently large initial data that yields a finite-time blowup solution (see \autoref{BlowupMSdgamma} below).
\end{remark}
\begin{remark}
The majorizing kernel $h:\R\to(0,\infty]$ does not exist in the case $d=1$, $\gamma=1$ \cite{orum-ossiander}, so our probabilistic model does not apply to this case. Sverak and Polacik showed that the complex Burgers equation ($d=1$, $\gamma=1$) has no minimal blowup data in the critical space $L^1(\R\times (0,1))$ \cite{sverakpolacik}.
\end{remark}
\begin{cor}\label{pmsol}
Let $d\ge 2$, $\gamma\in(\frac{1}{2},\frac{d+2}{4})$, and let $a_*$ and $\kappa$ be defined as in \autoref{mspmsol}. We have the following statements.
\begin{enumerate}[(i)]
\item If $u_0\in PM^a$ for $a\in(a_*,d)$ then there exists a number $\tau\in(0,\infty)$ such that \eqref{uhateq} has a unique solution $u\in BC((0,\tau),PM^a)$.
\item If $u_0\in PM^{a_*}$ and $\|u_0\|_{PM^{a_*}}\le 2\kappa$, then \eqref{uhateq} has a global solution $u\in BC((0,\infty),PM^{a_*})$.
\item If $u_0\in PM^{a_*}$ and $\|u_0\|_{PM^{a_*}}< 2\kappa$, the global solution $u$ is unique in the ball $\{w:\|w\|<2\kappa\}$ of $L^\infty((0,\infty),PM^{a_*})$. Moreover,
\[|\hat{u}(\xi,t)|\le c_1e^{-\delta|\xi|t^{1/(2\gamma)}}|\xi|^{-a_*}\ \ \ \forall(\xi,t)\]
where the constants $c_1,\delta>0$ depend on $\gamma$ and $\|u_0\|_{PM^{a_*}}$.
\end{enumerate}
\end{cor}
\begin{proof}
Part (i) was proven in \cite[Theorem 1.2, 1.3]{ferreira}. For Part (ii), let $\chi_0(\xi)=c_0\frac{\hat{u}_0(\xi)}{h(\xi)}$ and $\psi_0=|\chi_0|/2$. Then $\|\psi_0\|_{L^\infty}=\frac{1}{2}\|\chi_0\|_{L^\infty}\le 1$. According to \autoref{mspmsol}(ii), \eqref{nFMS} with initial data $\psi_0$ has a global solution $\psi=\EXP\barX$, where $\barX$ is the minimal solution process defined by \eqref{tilX}, and $0\le \psi(\xi,t)\le 1$ for all $(\xi,t)$. By \autoref{69211}(c), the function $\chi=\EXP_\xi\bX$ is a global solution to \eqref{mildnse} and $|\chi|\le 2\psi\le 2$. Here, $\bX$ is the minimal solution process defined by \eqref{eq:822191}.

For Part (iii), let $u$ and $v$ be functions such that $\hat{u}=\chi h/c_0$ and $\hat{v}=\psi h/c_0$. Since $\|\psi_0\|_{L^\infty}<1$, by \autoref{mspmsol}(iii) one has
\[\hat{v}(\xi,t)\le ce^{-\delta|\xi|t^{1/(2\gamma)}}|\xi|^{-a_*}\ \ \ \forall(\xi,t)\]
or equivalently,
\[\psi(\xi,t)\lesssim e^{-\delta|\xi|t^{1/(2\gamma)}}\ \ \ \forall(\xi,t).\]
By the Majorizing Principle, $|\chi(\xi,t)|\le 2\psi(\xi,t)\lesssim e^{-\delta|\xi|t^{1/(2\gamma)}}$. 

To prove that $u$ is the unique solution to \eqref{uhateq} in the ball $\{w:\|w\|<2\kappa\}$ of $L^\infty((0,\infty),PM^{a_*})$, we use the ground-state argument as was done in \autoref{mspmsol}(iii). Let $\tilde{u}$ be another solution to \eqref{uhateq} such that $\|\tilde{u}\|_{L^\infty_t PM^{a_*}}<2\kappa$. Let $\tilde{\chi}=c_0\mathcal{F}\{\tilde{u}\}/h$, which is a solution to \eqref{mildnse}, and let $$\alpha=\max\{\|\chi_0\|_{L^\infty},\|\tilde{\chi}\|_{L^\infty}\}<2.$$
Consider a sequence of stochastic processes $\tilde{\bX}_n$ defined inductively by $\tilde{\bX}_0(\xi,t)=\tilde{\chi}(\xi,t)$ and for $n\ge 0$,
\begin{equation}\label{tilX2}\tilde{\bX}_{n+1}(\xi,t)=\left\{ \begin{array}{*{35}{l}}
   \chi_0(\xi) & \text{if} & Y_\root> t,  \\
   \tilde{\bX}_n^{(1)}(W_1, t-Y_\root)\odot_{\xi}\tilde{\bX}_n^{(2)}(W_2, t-Y_\root) & \text{if} & Y_\root\le t.  \\
\end{array} \right.\end{equation}
By conditioning on $Y_\root$, one has
\[\EXP[\tilde{\bX}_{n+1}(\xi,t)]={{e}^{-t|\xi {{|}^{2\gamma}}}}(1+\epsilon)+\int_{0}^{t}{{{e}^{-s|\xi {{|}^{2\gamma}}}}|\xi {{|}^{2\gamma}}\int_{{{\mathbb{R}}^{d}}}{\EXP[\tilde{\bX}_n (\eta ,t-s)]\odot_\xi\EXP[\tilde{\bX}_n (\xi -\eta ,t-s)]H(\eta |\xi )d\eta ds }}.\]
Clearly, $\EXP\tilde{\bX}_0=\tilde{\chi}$. By induction on $n$, one get $\EXP\tilde{\bX}_n=\tilde{\chi}$ for all $n\ge 0$. On the non-explosion event $[S>t]$, one has $\tilde{\bX}_n=\bX$ for sufficiently large $n$, and $\bX$ is an $\odot$-product of $|\partial V(\xi,t)|$ vectors, each of which has length of at most $\alpha$. By the inequality $|a\odot_\xi b|\le\frac{1}{2}|a||b|$, one has
\[|\bX(\xi,t)|\1_{S>t}\le 2^{1-|\partial V(\xi,t)|}\alpha^{|\partial V(\xi,t)|}\le 2.\]
On the explosion event $[S\le t]$, $\tilde{\bX}_n$ is an $\odot$-product of $2^n$ vectors, each of which has length of at most $\alpha$. Thus, one has $$|\tilde{\bX}_n(\xi,t)|\1_{S\le t}\le 2^{1-2^n}\alpha^{2^n}\le 2$$
and thus, $\tilde{\bX}_n(\xi,t)\1_{S\le t}\to 0$ as $n\to\infty$. By Lebesgue's Dominated Convergence Theorem, $\lim\EXP\tilde{\bX}_n=\EXP\bX$. In other words, $\tilde{\chi}=\chi$.
\end{proof}
In \cite[Prop.\ 4.14]{dascaliuc2019jan}, where the case $d\ge 3$ and $\gamma=1$ is considered, it is shown that the solution blows up in finite time for sufficiently large compactly supported initial data. In the following proposition, we show that the same result holds for the general case $d\ge 1$ and $\gamma\in(\frac{1}{2},\frac{d+2}{4})$. Note that an initial data $u_0$ compactly supported in the Fourier space yields a local-in-time solution $u\in BC((0,\tau_a),PM^a)$ for any $a\in(a_*,d)$ according to \autoref{mspmsol}(i).
\begin{prop}
\label{BlowupMSdgamma}
Let $d\ge 1$ and $\gamma\in(\frac{1}{2},\frac{d+2}{4})$. Let $u_0:\mathbb{R}^d\to\mathbb{R}$ be a function 
whose Fourier transform is real-valued, nonnegative on $\mathbb{R}^d\backslash\{0\}$, and bounded away from zero on some nonempty open subset of $\mathbb{R}^d$. Then there exists $a_0>0$ such that for any $a\ge a_0$, the solution $u$ to the Cauchy problem {\eqref{FMS}} with initial data $au_0$ satisfies $u(\cdot,1)\not\in \dot{B}^\alpha_{\infty,\infty}$ for any $\alpha\in\mathbb{R}$.
\end{prop}
\begin{proof}
Without loss of generality, one may assune $\hat{u}_0(\xi)\ge \epsilon>0$ for all $\xi\in D=(2e_1+B_1)\cup(-2e_1+B_1)$ where $B_1$ is the unit ball centered at the origin and $e_1=(1,0,...,0)\in\mathbb{R}^d$. 
Denote 
\[q_n(\xi,t)=\mathbb{P}_{\xi}(S_\xi>t,\ |\partial V(\xi,t)|=n,\ W_v\in D \text{ for all } v\in\partial V(\xi,t)\}.\] 
Let $\psi_0=c_0a\hat{u}_0/h$ and let $\psi=\EXP\bar{\X}$ be the solution to \eqref{nFMS} with initial condition $\psi_0$. By \eqref{constantc} and the fact that $|\xi|\ge 1$ for all $\xi\in D$, one has $h(\xi)\le c_{d,\gamma}$ for all $\xi\in D$. Then
\begin{equation}\label{eq:928212}\psi (\xi ,t)=\EXP\left[\prod_{v\in\partial V(\xi,t)}\psi_0(W_v)\1_{S>t}\right]\ge\EXP\left[\left(\frac{ac_0\epsilon}{c_{d,\gamma}}\right)^{|\partial V(\xi,t)|}\1_{S>t}\right]\ge \sum\limits_{n=1}^{\infty }{{\left(\frac{ac_0\epsilon}{c_{d,\gamma}}\right)^{n}}{{q}_{n}}(\xi ,t)}.\end{equation}
We have $q_1(\xi,t)=e^{-|\xi|^{2\gamma}t}\mathbbm{1}_D(\xi)$. By conditioning on the first time of branching, one gets
\begin{equation}\label{probrecursion2}
{{q}_{n}}(\xi ,t)=\int_{0}^{t}{|\xi |^{2\gamma}{{e}^{-s|\xi|^{2\gamma}}}\int_{{{\mathbb{R}}^{d}}}{\sum\limits_{k=1}^{n-1}{{{q}_{k}}(\eta ,t-s){{q}_{n-k}}(\xi -\eta ,t-s)H(\eta |\xi )d\eta ds}}}.\end{equation}
From now on, the proof follows almost verbatim the proof of Proposition 4.14 in \cite{dascaliuc2019jan}, which we will skip.
\end{proof}
\section{Equations in $\R^2$ - Role of Symmetries in Extending Solutions Beyond $t_c$.}\label{Section5}
In $\R^2$, the orthogonality constraint $\xi\cdot\chi(\xi)=0$ imposes that $\chi(\xi)$ is a scaling of $\xi^\perp$. This geometric property gives a detailed structure of the $\odot$-product as we already pointed out in \autoref{cased=2}. In the following proposition, we establish a closed form of the minimal solution process. 

Recall that $\mathring{V}(t)$ is the set of $t$-internal vertices defined by \eqref{tancestors}. The relation $|\Vo(\xi,t)|=|\partial V(\xi,t)|-1$ holds on the non-explosion event $[S>t]$. Denote by $\theta_{u,v}\in[0,\pi]$ the angle between two vectors $u$ and $v$.



\begin{prop}\label{Xclosedform}
Let $d=2$ and $\gamma\in(1/2,1)$, and let $\chi_0:\R^2\to \C^2$ be a vector field satisfying \eqref{divfree}. Let $\bX(\xi,t)=\bX(\xi,t)\1_{[S>t]}$ be the minimal solution process defined by \eqref{eq:822191}. One has
\begin{eqnarray}
\nonumber\X(\xi,t)=\left(\frac{i}{2}\right)^{|\Vo(\xi,t)|}\left(\prod_{v\in \Vo(\xi,t)}\textup{sign}(W_v\times W_{v1})\sin(\theta_{W_v,W_{v1}}-\theta_{W_v,W_{v2}})\right)\cdot\\
\label{closedform}\left(\prod_{v\in\partial V(\xi,t)}\chi_0(W_v)\cdot e_{W_v^\perp}\right)e_{\xi^\perp}
\end{eqnarray}
\end{prop}
\begin{proof}
We will prove by induction on $N(\xi,t)=|\Vo(\xi,t)|$. If $N(\xi,t)=0$ then $\Vo(\xi,t)=\emptyset$, $\partial V(\xi,t)=\{\root\}$ and
\[\X(\xi,t)=\chi_0(\xi,t)=(\chi_0(\xi,t)\cdot e_{\xi^\perp})e_{\xi^\perp}.\]
Thus, \eqref{closedform} holds for the base case $N(\xi,t)=0$. Suppose that \eqref{closedform} holds for $N(\xi,t)\le n$ for some $n\ge 0$. Consider the case $N(\xi,t)=n+1\ge 1$. Then $Y_\root<t$ and
\begin{equation}\label{52251}\X(\xi,t)=\X^{(1)}(W_1,t-Y_\root)\odot_{\xi}\X^{(2)}(W_2,t-Y_\root).\end{equation}
Note that
\begin{eqnarray}
\label{52252}\Vo(\xi,t)&=&\{\root\}\cup\Vo^{(1)}(W_1,t-Y_\root)\cup\Vo^{(2)}(W_2,t-Y_\root)\\
\label{52253}\partial V(\xi,t)&=&\partial V^{(1)}(W_1,t-Y_\root)\cup\partial V^{(2)}(W_2,t-Y_\root)
\end{eqnarray}
where $\Vo^{(k)}$ and $\partial V^{(k)}$ are the set of internal vertices and the set of leaves of the subtree rooted at $W_k$. Since $N^{(1)}(W_1,t-Y_\root)+N^{(2)}(W_2,t-Y_\root)=N(\xi,t)-1=n$, one deduces by the induction hypothesis that for $k\in\{1,2\}$,
\begin{eqnarray*}
\nonumber\X^{(k)}(W_k,t-Y_\root)&=&\left(\frac{i}{2}\right)^{N^{(k)}(W_k,t-Y_\root)}\cdot\\
&&\left(\prod_{v\in \Vo^{(k)}(W_k,t-Y_\root)}\textup{sign}(W_v\times W_{v1})\sin(\theta_{W_v,W_{v1}}-\theta_{W_v,W_{v2}})\right)\cdot\\
&&\left(\prod_{v\in\partial V^{(k)}(W_k,t-Y_\root)}\chi_0(W_v)\cdot e_{W_v^\perp}\right)e_{W_k^\perp}
\end{eqnarray*}
Applying \eqref{52251}-\eqref{52253} and the following identity (proven in \autoref{cased=2})
\[e_{W_1^\perp}\odot_\xi e_{W_2^\perp}=\frac{i}{2}\text{sign}(\xi\times W_1)\sin(\theta_{\xi,W_1}-\theta_{\xi,W_2})e_{\xi^\perp},\]
one obtains \eqref{closedform}.
\end{proof}
Thanks to \autoref{Xclosedform}, we can write the minimal solution process as
\begin{equation}\label{funcphi}
\bX(\xi,t)=\X(\xi,t)\1_{[S>t]}=\phi_{\xi,t}(\{\Sigma_v\}_{v\in\T},\{\Theta_v\}_{v\in\T})\end{equation}
where the family of \emph{signs} $\{\Sigma_v\}_{v\in\T}$ and the family of \emph{angles} $\{\Theta_v\}_{v\in\T}$ are described in the following proposition.
\begin{prop}\label{iidfamilies}
For any $v\in\T$, let $\Sigma_v=\textup{sign}(W_v\times W_{v1})$ and
\[\Theta_{v}=(\theta_{W_v,W_{v1}},\theta_{W_v,W_{v2}})\in\triangle:=\{(\phi_1,\phi_2):\ \phi_1,\phi_2\ge 0,\ \phi_1+\phi_2\le\pi\}.\]
Then
\begin{enumerate}[(a)]
\item $\{\Sigma_v\}_{v\in \T}$ is an i.i.d.\ family of Bernoulli random variables with parameter $1/2$.
\item $\{\Theta_v\}_{v\in \T}$ is an i.i.d.\ family of random vectors whose joint pdf is given by \eqref{angle-distr}.
\item $\{\Sigma_v\}_{v\in \T}$ and $\{\Theta_v\}_{v\in \T}$ are independent each other.
\end{enumerate}
\end{prop}
\begin{proof}
To prove (a), notice that the vector $W_v$ partitions the plane into two halves
\begin{eqnarray*} 
A&=&\{\eta\in\mathbb{R}^2:\ W_v\times \eta=W_v^\perp\cdot \eta>0\},\\
B&=&\{\eta\in\mathbb{R}^2:\ W_v\times \eta=W_v^\perp\cdot \eta<0\}.
\end{eqnarray*}
We have
\begin{eqnarray*}
\P_\xi(\Sigma_v=1)=\P_\xi(W_{v1}\in A)&=&\EXP_\xi[\P_\xi(W_{v1}\in A|W_v)]=\EXP_\xi\left[\int_A H(\eta|W_v)d\eta\right].
\end{eqnarray*}
Using the change of variable $\zeta=W_v-\eta$ and noting that $H(\eta|W_v)=H(\zeta|W_v)$, we have
\[\P_\xi(\Sigma_v=1)=\EXP_\xi\left[\int_A H(\zeta|W_v)d\eta\right]=\EXP_\xi\left[\int_B H(\eta|W_v)d\eta\right]=\P_\xi(W_{v1}\in B)=\P_\xi(\Sigma_v=-1).\]
Therefore, $\Sigma_v\sim \text{Bernoulli}(1/2)$.

Part (b) and (c) are obvious since each vector $\Theta_v$ has a joint pdf given by \eqref{angle-distr}.
\end{proof}

The structure of the $\odot$-product in $\R^2$ as given in \eqref{odot} also gives a stronger characterization of the critical time of integrability $t_c$ defined by \eqref{tcritical}. Specifically, one has
\begin{prop}\label{tprop}
Let $\chi_0:\R^2\to \C^2$ be a vector field satisfying \eqref{divfree} and $\bX$ be the minimal solution process defined by \eqref{eq:822191}. 
Let $t_c\ge 0$ be defined by \eqref{tcritical} and  $f(\xi,t)=\EXP_\xi[|\X(\xi,t)|\1_{S>t}]\in[0,\infty]$. Then: 
\begin{enumerate}[(a)]
\item The function $v(\xi,t)=\frac{h(\xi)}{c_{0}}f(\xi,t)$ satisfies
\begin{equation}\label{mildfnsscalar}
v(\xi,t)=v_0(\xi)e^{-|\xi|^{2\gamma}t}+\frac{c_0}{2}\int_0^t \int_{\R^2} |\xi|e^{-|\xi|^{2\gamma}s} v(\eta,t-s) v(\xi-\eta,t-s)|\sin(\theta_{\xi,\eta}-\theta_{\xi,\xi-\eta})| d\eta ds
\end{equation}
where $v_0(\xi)=\frac{h(\xi)}{c_0}|\chi_0(\xi)|$. Consequently, $e^{|\xi|^{2\gamma}t}f(\xi,t)$ is increasing in $t$ for every fixed $\xi\in\R^2\backslash\{0\}$.
\item If $t_c>0$ then for every $t\in(0,t_c)$, $f(\xi,t)<\infty$ for almost every $\xi\in\R^2$. Consequently, $f(\xi,t)<\infty$ for almost every $(\xi,t)\in\mathbb{R}^2\times(0,t_c)$.
\item If
\begin{equation}\label{sym}
\exists A\subset \R^2: m(A)>0 \text{~~and~~}\chi_0(\xi)\ne 0\text{~~a.e.~~}\xi\in A\cup (-A)\end{equation}
then $f(\xi,t)>0$ for all $\xi\in\R^2\backslash\{0\}$ and $t>0$. Here, $m$ denotes the Lebesgue measure on $\R^2$.
\item If $t_c<\infty$ and \eqref{sym} is satisfied then  $f(\xi,t)=\infty$ for all $(\xi,t)\in\R^2\times(t_c,\infty)$.
\end{enumerate}
\end{prop}
\begin{proof}
(a) By \eqref{eq:822191} and \eqref{odot},
\begin{equation}\label{|X|}|\X(\xi,t)|=\left\{ \begin{array}{*{35}{l}}
   |\chi_0(\xi)| & \text{if} & {{Y}_\root}> t,  \\
   \frac{1}{2}|{\X}^{(1)}(W_1, t-{{Y}_\root})||{\X}^{(2)}(W_2, t-{{Y}_\root})||\sin(\theta_{\xi,W_1}-\theta_{\xi,W_2})| & \text{if} & {{Y}_\root}\le t.  \\
\end{array} \right.\end{equation}
Taking the expectation of \eqref{|X|}, one obtains \eqref{mildfnsscalar}.\\\\
(b) Let
\begin{eqnarray}
\nonumber E&=&\{(\xi,t)\in\R^2\times(0,t_c):f(\xi,t)=\infty\},\\
\label{et}E_t&=&\{\xi\in\R^2:f(\xi,t)=\infty\}.
\end{eqnarray}
By the monotonicity of $e^{|\xi|^{2\gamma}t}f(\xi,t)$ as proven in Part (a), $\{E_t\}_{t>0}$ is an increasing family. That is, $E_t\subset E_s$ if $t<s$. 
Suppose by contradiction that $m(E_{t_0})>0$ for some $t_0\in(0,t_c)$. The monotonicity of $\{E_t\}_{t>0}$ implies $m(E_s)>0$ for all $s\in(t_0,t_c)$. This contradicts the definition of $t_c$. Therefore, $m(E_t)=0$ for all $t\in(0,t_c)$. By Fubini's Theorem, we have
\[m(E)=\int_0^{t_c}\int_{\R^2}\1_E(\xi,t)d\xi dt=\int_0^{t_c}m(E_t)dt=0.\]
(c) Assume \eqref{sym}. Since $A$ is of positive measure, there exists a number $M>0$ such that the set $B=\{\xi\in A\cup(-A):|\xi|\le M\}$ is also of positive measure. For any $\zeta\in\R^2$, denote
\[\zeta+B=\{\zeta+\xi:\xi\in B\}.\]
For each $r>0$, let $D_r=\{\xi\in\R^2:|\xi|<r\}$. By the continuity of the translation operator in $L^1(\R^2)$ (see e.g.\ \cite[Exer.\ 12, p.\ 132]{rudin}),
\[\lim_{\zeta\to 0}\int_{\R^2}|\1_{\zeta+B}(\xi)-\1_{B}(\xi)|d\xi=0.\]
Equivalently, $m((\zeta+B)\backslash B)+m(B\backslash(\zeta+B))\to 0$ as $\zeta\to 0$. This implies $m(B\cup (\zeta+B))=m(B)+m((\zeta+B)\backslash B)\to m(B)$ as $\zeta\to 0$. Hence,
\[m(B\cap(\zeta+B))=m(B)+m(\zeta+B)-m(B\cup(\zeta+B))=2m(B)-m(B\cup(\zeta+B))\to m(B)\]
as $\zeta\to 0$. Thus, there exists $\epsilon>0$ such that $m(B\cap(\zeta+B))>0$ for all $\zeta\in\R^2,|\zeta|<\epsilon$. Let
\[\tilde{E}=\{\xi\in\R^2:v(\xi,t)>0\ \ \forall\, t>0\}.\]
To show that $\tilde{E}=\R^2$, we will show by induction on $n\in\mathbb{N}$ that $D_{n\epsilon}\subset \tilde{E}$.

For any $\xi\in D_\epsilon$ and $t>0$, \eqref{mild|nse|} implies
\begin{equation}\label{73251}v(\xi,t)\ge \frac{c_0}{2}\int_0^t \int_{B\cap(\xi+B)} |\xi|e^{-|\xi|^{2\gamma}s} v(\eta,t-s) v(\xi-\eta,t-s)|\sin(\theta_{\xi,\eta}-\theta_{\xi,\xi-\eta})| d\eta ds.\end{equation}
For any $\eta\in B\cap(\xi+B)$, we have $\xi-\eta\in B\cap(\xi+B)$ and thus, $v_0(\eta)>0$ and $v_0(\xi-\eta)>0$. Then $v(\eta,s)\ge v_0(\eta)e^{-|\eta|^{2\gamma}s}>0$ and $v(\xi-\eta,s)\ge v_0(\xi-\eta)e^{-|\xi-\eta|^{2\gamma}s}>0$ for any $s>0$. Then \eqref{73251} implies that $v(\xi,t)>0$, which implies $D_\epsilon\subset \tilde{E}$. 

Now suppose that $D_{n\epsilon}\subset \tilde{E}$ for some $n\in\mathbb{N}$. Let $\xi\in D_{(n+1)\epsilon}$. Note that the set $D=D_{n\epsilon}\cap (\xi+D_\epsilon)$ is of positive measure. By \eqref{mildfnsscalar},
\begin{equation}\label{73252}v(\xi,t)\ge \frac{c_0}{2}\int_0^t \int_{D} |\xi|e^{-|\xi|^{2\gamma}s} v(\eta,t-s) v(\xi-\eta,t-s)|\sin(\theta_{\xi,\eta}-\theta_{\xi,\xi-\eta})| d\eta ds.\end{equation}
For any $\eta\in D$, we have $\eta\in D_{n\epsilon}\subset\tilde{E}$. On the other hand, the fact that $\eta\in (\xi+D_{\epsilon})$ implies $\xi-\eta\in D_\epsilon \subset \tilde{E}$. Hence, $v(\eta,s)>0$ and $v(\xi-\eta,s)>0$ for any $s>0$. By \eqref{73252}, $v(\xi,t)>0$. We have shown that $D_{(n+1)\epsilon}\subset \tilde{E}$.\\\\
(d) Suppose that $t_c<\infty$ and that \eqref{sym} is satisfied. Fix $t>t_c$ and let $\tau=(t+t_c)/2$. Since the family $\{E_s\}_{s>0}$ defined by \eqref{et} is increasing and $t_c<\tau<\infty$, we have $m(E_\tau)>0$. On the other hand, by Part (c), one has $v(\zeta,s)>0$ for any $\zeta\in\R^2$ and $s>0$. Hence, for any $\xi\in\R^2$ and $\eta\in E_\tau$, we have $v(\eta,s)=\infty$ and $v(\xi-\eta,s)>0$ for any $\tau\le s\le t$. Consequently,
\[v(\xi,t)\ge \frac{c_0}{2}\int_0^{t-\tau} \int_{E_\tau} |\xi|e^{-|\xi|^{2\gamma}s} v(\eta,t-s) v(\xi-\eta,t-s)|\sin(\theta_{\xi,\eta}-\theta_{\xi,\xi-\eta})| d\eta ds=\infty.\]
\end{proof}
In the following, we show that $t_c$ is finite for sufficiently large initial data.
\begin{prop}\label{blowup}
Let $d=2$ and $\gamma\in(1/2,1)$, and let $u_0:\R^2\to \R^2$ be a divergence-free vector field. Let $\chi_0(\xi)=c_{0}\frac{\hat{u}_0(\xi)}{h(\xi)}$ and $\bX$ be the minimal solution process defined by \eqref{eq:822191}. Suppose
\[M=\inf_{4\le|\xi|\le 7}|\hat{u}_0(\xi)|>500\cdot 7^{2\gamma}.\]
Then
\[t_c\le T_*:=-7^{-2\gamma}\ln\left(1-\frac{500\cdot 7^{2\gamma}}{M}\right).\]
\end{prop}
\begin{proof}
Let $v(\xi,t)=\frac{h(\xi)}{c_{0}}\EXP[|\X(\xi,t)|\1_{[S>t]}]$ and $v_0(\xi)=|\hat{u}_0(\xi)|$. By \autoref{tprop}(a), one gets that for all $t>0$,
\begin{equation}\label{mild|nse|}
v(\xi,t)=v_0(\xi)e^{-|\xi|^{2\gamma}t}+\frac{c_0}{2}\int_0^t \int_{\R^2} |\xi|e^{-|\xi|^{2\gamma}s} v(\eta,t-s) v(\xi-\eta,t-s)|\sin(\theta_{\xi,\eta}-\theta_{\xi,\xi-\eta})| d\eta ds.
\end{equation}
Note that $v_0(\xi)\ge M$ for $4\le|\xi|\le 7$. Suppose by contradiction that $t_c>T_*$. By \autoref{tprop}(b),  we have that for each $t\in(0,T_*)$, 
\begin{equation}\label{56255}v(\xi,t)<\infty,\ \ \ \text{a.e.}\ \xi\in\R^2.
\end{equation}
Let
\[A_\xi=\{\eta\in\R^2:\ 4\le|\eta|\le 5,\ 6\le|\xi-\eta|\le 7\}.\]
Then for any $\xi$ such that $4\le|\xi|\le 7$, one has
\begin{equation}\label{56251}
v(\xi,t)\ge e^{-7^{2\gamma}t}M+\int_0^t \int_{A_\xi} \frac{1}{\pi}e^{-7^{2\gamma}s} v(\eta,t-s) v(\xi-\eta,t-s)|\sin(\theta_{\xi,\eta}-\theta_{\xi,\xi-\eta})| d\eta ds.
\end{equation}
Let $\beta=7^{2\gamma}$. Multiplying both sides of \eqref{56251} by $e^{\beta t}$, we get
\begin{equation}\label{56252}
e^{\beta t}v(\xi,t)\ge M+\int_0^t \int_{A_\xi} \frac{1}{\pi}e^{\beta(t-s)} v(\eta,t-s) v(\xi-\eta,t-s)|\sin(\theta_{\xi,\eta}-\theta_{\xi,\xi-\eta})| d\eta ds
\end{equation}
for any $\xi$ such that $4\le|\xi|\le 7$. Next, we find lower bounds for $|A_\xi|$ and $|\sin(\theta_{\xi,\eta}-\theta_{\xi,\xi-\eta})|$ provided that 
\[4\le|\xi|\le 7,\ 4\le|\eta|\le 5,\ 6\le|\xi-\eta|\le 7.\]

\subsubsection*{Estimate of $|\sin(\theta_{\xi,\eta}-\theta_{\xi,\xi-\eta})|$:}

Let $\phi_1=\sin\theta_{\xi,\eta}\in[0,\pi]$, $\phi_2=\sin\theta_{\xi,\xi-\eta}\in[0,\pi]$, $a=|\eta|$, $b=|\xi-\eta|$, $c=|\xi|$. Label points $A,B,C,D$ as in \autoref{triangle}.
\begin{figure}[h!]
\centering\includegraphics[scale=.8]{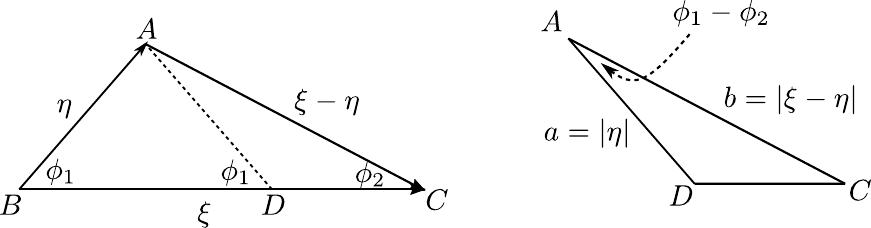}
\caption{}
\label{triangle}
\end{figure}
By the triangle inequality, $DC>AC-AD=|\xi-\eta|-|\eta|\ge 6-5=1$. By the Sine Law in triangle $ADC$,
\[|\sin(\phi_1-\phi_2)|=\frac{2\text{area}(ADC)}{ab}\ge\frac{2\text{area}(ADC)}{5\times 7}=\frac{2\text{area}(ADC)}{35}.\]
On the other hand,
\[\frac{\text{area}(ADC)}{\text{area}(ABC)}=\frac{DC}{BC}\ge\frac{1}{7}.\]
Combining the two above inequalities, one gets
\[|\sin(\phi_1-\phi_2)|\ge\frac{2}{245}\text{area}(ABC)\]
By Heron's formula,
\[\text{area}(ABC)=\frac{1}{4}\sqrt{(a+b+c)(a+b-c)(a+c-b)(b+c-a)}.\]
Under the conditions $4\le a\le 5$, $6\le b\le 7$, $4\le c\le 7$, we have
\[\text{area}(ABC)\ge\frac{1}{4}\sqrt{(4+6+4)(4+6-7)(4+4-7)(6+4-5)}=\frac{\sqrt{210}}{4}.\]
Thus,
\begin{equation}\label{56253}|\sin(\phi_1-\phi_2)|\ge\frac{2}{245}\frac{\sqrt{210}}{4}=\frac{\sqrt{210}}{490}>\frac{1}{40}.\end{equation}

\subsubsection*{Estimate of $|A_\xi|$:}

For the purpose of estimating $|A_\xi|$, we can assume that $\xi$ has coordinates $(c=|\xi|,0)$. Label the points $O(0,0)$, $F(c,0)$, $E\left(-\frac{12}{c}+\frac{c}{2},\sqrt{20-\left(-\frac{12}{c}+\frac{c}{2}\right)^2}\right)$ as in \autoref{circles}.
\begin{figure}[h!]
\centering\includegraphics[scale=.7]{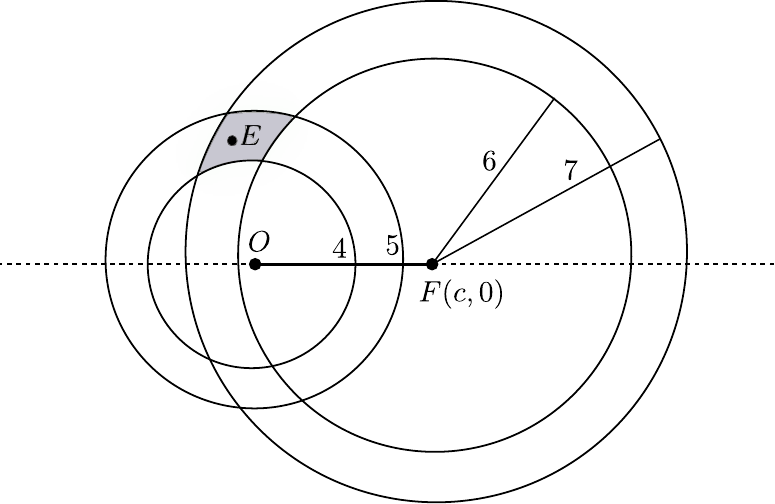}
\caption{}
\label{circles}
\end{figure}
Then
\begin{eqnarray*}
OE&=&\sqrt{\left(-\frac{12}{c}+\frac{c}{2}\right)^2+20-\left(-\frac{12}{c}+\frac{c}{2}\right)^2}=\sqrt{20},\\
EF&=&\sqrt{\left(-\frac{12}{c}-\frac{c}{2}\right)^2+20-\left(-\frac{12}{c}+\frac{c}{2}\right)^2}=\sqrt{44}.
\end{eqnarray*}
For any $\eta\in B_{3/10}(E)$,
\begin{eqnarray*}
|\eta|&\le& OE+|\overrightarrow{OE}-\eta|\le\sqrt{20}+\frac{3}{10}<5,\\
|\eta|&\ge& OE-|\overrightarrow{OE}-\eta|\le\sqrt{20}-\frac{3}{10}>4,\\
|\xi-\eta|&\le& EF+|\overrightarrow{OE}-\eta|\le\sqrt{44}+\frac{3}{10}<7,\\
|\xi-\eta|&\ge& EF-|\overrightarrow{OE}-\eta|\le\sqrt{44}-\frac{3}{10}>6.\\
\end{eqnarray*}
Therefore, $B_{3/10}(E)\subset A_\xi$. Consequently,
\begin{equation}\label{56254}
|A_\xi|\ge |B_{3/10}|=\pi\left(\frac{3}{10}\right)^2=\frac{9\pi}{100}.
\end{equation}
Substituting the estimate \eqref{56253} into \eqref{56252}, we get
\begin{eqnarray*}
e^{\beta t}v(\xi,t)\ge M+\frac{1}{40\pi}\int_0^t \int_{A_\xi} e^{\beta(t-s)} v(\eta,t-s) v(\xi-\eta,t-s)|\sin(\theta_{\xi,\eta}-\theta_{\xi,\xi-\eta})| d\eta ds.
\end{eqnarray*}
Let $q(t)=\inf_{4\le |\xi|\le 7}e^{\beta t}v(\xi,t)$. Because of \eqref{56255}, $q(t)<\infty$ for all $t\in(0,T_*)$. The inequality above together with \eqref{56254} implies
\begin{eqnarray*}
q(t)&\ge& M+\frac{1}{40}\int_0^t \int_{A_\xi} e^{-\beta(t-s)} q^2(t-s) d\eta ds\\
&=& M+\frac{1}{40\pi}|A_\xi|\int_0^t e^{-\beta(t-s)} q^2(t-s) ds\\
&\ge& M+\frac{1}{450}\int_0^t e^{-\beta s} q^2(s) ds
\end{eqnarray*}
Then $q(t)\ge p(t)$, where $p(t)$ is the solution to the integral equation
\[p(t)=M+\frac{1}{450}\int_0^t e^{-\beta s} p^2(s) ds.\]
Differentiating both sides, we get $p'=\frac{1}{450}e^{-\beta t}p^2$, which has an explicit solution
\[p(t)=\frac{1}{\frac{1}{M}-\frac{1}{450\beta}(1-e^{-\beta t})}.\]
This function blows up at
\[\tau=-\frac{1}{\beta}\ln\left(1-\frac{450\beta}{M}\right).\]
Notice that $\tau<T_*$. For any $t<\tau$, we have
\begin{eqnarray*}
q(\tau)&\ge& M+\frac{1}{450}\int_0^{\tau} e^{-\beta s} q^2(s) ds\\
&\ge& M+\frac{1}{450}\int_0^{t} e^{-\beta s} q^2(s) ds\\
&\ge& M+\frac{1}{450}\int_0^{t} e^{-\beta s} p^2(s) ds\\
&=&p(t).
\end{eqnarray*}
Thus, $q(\tau)\ge \limsup_{t\to\tau^-} p(t)=\infty$. This is a contradiction.
\end{proof}
The fact that the solution process $\X(\xi,t)$ fails to be integrable after finite time for sufficiently large initial condition does not necessarily imply that the mild solution $u$ to \ref{nsedg} blows up after finite time. It only implies that the probabilistic representation of the solution to the scaled Fourier transform $\chi=c_{0}\frac{\hat{u}}{h}$ ceases to be valid after finite time. It is possible, given a certain symmetry of the initial data that enables stochastic cancellations, to extend the solution beyond the critical time $t_c$. The rest of the section is to demonstrate this idea.

Recall the family of signs $\{\Sigma_v\}_{v\in\T}$ and the family of angles $\{\Theta_v\}_{v\in\T}$ given in \autoref{iidfamilies}. Consider a transformation $(\{\Sigma_v\}_{v\in\T},\{\Theta_v\}_{v\in\T})\mapsto (\{\tilde{\Sigma}_v\}_{v\in\T},\{\tilde{\Theta}_v\}_{v\in\T})$ in which 
\begin{eqnarray}
\label{trans1}\tilde{\Theta}_v&=&\Theta_v\ \ \ \forall v\in \T,\\
\label{trans2}{{\tilde{\Sigma }}_{v}}&=&\left\{ \begin{array}{*{35}{rcl}}
   -{{\Sigma }_{v}} & \text{if} & v=\root   \\
   {{\Sigma }_{v}} & \text{if} & v\in \T\backslash \{\root \}.  \\
\end{array} \right.
\end{eqnarray}
This transformation induces a transformation 
$\{W_v\}_{v\in\T}\to \{\tilde{W}_v\}_{v\in\T}$ in which 
$\tilde{W}_\emptyset=W_\emptyset=\xi$ and $|\tilde{W}_v|=|W_v|$ for all $v\in\T$. 
The DSY cascade $\{Y_v=|W_v|^{-2\gamma}T_v\}_{v\in\T}$ is then transformed into 
$\{\tilde{Y}_v=|\tilde{W}_v|^{-2\gamma}T_v=Y_v\}_{v\in\T}$. The explosion time 
is transformed into
\[\tilde{S}=\underset{n\ge 0 }{\mathop{\sup }}\,
\underset{|v|=n}{\mathop{\min }}\,\sum\limits_{j=0}^{n} \tilde{Y}_{v|j}
=\underset{n\ge 0 }{\mathop{\sup }}\,
\underset{|v|=n}{\mathop{\min }}\,\sum\limits_{j=0}^{n} {Y}_{v|j}=S.\]

To use {\autoref{iidfamilies} 
and the transformations \eqref{trans1}-\eqref{trans2} to explore the possibilities for stochastic cancellation, we consider the well-known simple vortex model, see e.g. \cite[p.\ 48]{majda}. A vortex flow is a flow in which the stream function is radially symmetric (i.e.\ all streamlines are concentric circles). In the Fourier space, a vortex flow $u_0$ has the form $\hat{u}_0(\xi)=U(|\xi|)(-\xi_2,\xi_1)^T$, or equivalently, 
\begin{equation}\label{vortex}
\chi_0(\xi)=g(|\xi|)e_{\xi^\perp}.
\end{equation}
It is known that a flow evolving from a vortex flow remains a vortex flow for all time. See e.g.\ \cite[Example 2.2, p.\ 48]{majda}, \cite[Sec.\ 4.5]{batchelor}, \cite[Chap.\ 13]{saffman}. We will show below that the same result holds for the fractional Navier-Stokes equations. 


\begin{prop}\label{signflip}
Let $\tilde{\X}(\xi,t)=\phi_{\xi,t}(\{\tilde{\Sigma}_v\}_{v\in\T},
\{\tilde{\Theta}_v\}_{v\in\T})$ where $\{\tilde{\Sigma}_v\}_{v\in\T}$ 
and $\{\tilde{\Theta}_v\}_{v\in\T}$ are given by \eqref{trans1}-\eqref{trans2}. 
Suppose 
the vector field $\chi_0$ is given by \eqref{vortex}. 
Then 
\begin{enumerate}[(a)]
\item $\X\1_{[T_\root<t<S]}=-\tilde{\X}\1_{[T_\root<t<{S}]}$.
\item If $\EXP[|\X(\xi,t)|\1_{[T_\root<t<S]}]<\infty$ then $\EXP[\X(\xi,t)\1_{[T_\root<t<S]}]=0$.
\end{enumerate}
\end{prop}
\begin{proof}
To show Part (a), notice from \eqref{closedform} that
\begin{equation}
\X\1_{S>t}=\left(\frac{i}{2}\right)^{|\Vo(\xi,t)|}\left(\prod_{v\in \Vo(\xi,t)}\Sigma_v\sin(\Theta_{v1}-\Theta_{v2})\right)\left(\prod_{v\in\partial V(\xi,t)}g(|W_v|)\right)e_{\xi^\perp}
\end{equation}
where $g(\eta)=\chi_0(\eta)\cdot e_{\eta^\perp}$. On the event $[T_\root<t<S]$, one has $\root\in\Vo(\xi,t)$. Using the fact that $\tilde{\Sigma}_\root=-\Sigma_\root$, one has
\begin{equation}
\tilde{\X}\1_{[T_\root<t<S]}=\left(\frac{i}{2}\right)^{|\Vo(\xi,t)|}\left(\prod_{v\in \Vo(\xi,t)}\tilde{\Sigma}_v\sin(\Theta_{v1}-\Theta_{v2})\right)\left(\prod_{v\in\partial V(\xi,t)}g(|W_v|)\right)e_{\xi^\perp}=-\X\1_{[T_\root<t<S]}.
\end{equation}
To show (b), notice that the family $\{\tilde{\Sigma}_v\}_{v\in\T}$ has the same distribution as $\{{\Sigma}_v\}_{v\in\T}$. Assuming the integrability condition $\EXP[|\X(\xi,t)|\1_{[T_\root<t<S]}]<\infty$, one has
\[\EXP[\X(\xi,t)\1_{[T_\root<t<S]}]=\EXP[\tilde{\X}(\xi,t)\1_{[T_\root<t<S]}].\]
Together with Part (a), one infers that $\EXP[\X(\xi,t)\1_{[T_\root<t<S]}]=0$.
\end{proof}


\begin{cor}\label{radial1}
Suppose $u_0:\R^2\to \R^2$ is a vortex. Then \ref{nsedg} has a global solution 
$u$ whose Fourier transform is given by 
$\hat{u}(\xi,t)=\hat{u}_0(\xi)e^{-|\xi|^{2\gamma}t}$.
\end{cor}

\begin{proof}
To exploit asymmetric cancellations, consider the simple
arithmetic average
$${\X}^* = \frac{1}{2}(\X + \tilde{\X}).$$
Clearly, $\hat{u}(\xi,t)= \mathbb{E}\X^*(\xi,t)=\EXP\bX(\xi,t)$ is a solution
for $t<t_c=t_c(\chi_0)$. In view of \autoref{signflip}, one 
has
$${\X}^*(\xi,t) = \hat{u}_0(\xi)\1_{[T_\root>t]}.$$
Thus, $u(\xi,t) = \hat{u}_0(\xi)e^{-|\xi|^{2\gamma}t}$
is a solution for $t< t_c$.  However, one may readily
check that the validity of this mild solution extends to all times
$t$ beyond $t_c$ as well.
\end{proof}

\bibliography{References}
\end{document}